\documentclass[a4paper]{amsart}

\usepackage[english]{babel}

\usepackage{amssymb}
\usepackage{amsthm}
\usepackage{mathtools}
\usepackage{dsfont}
\usepackage{nicefrac}
\usepackage{cases}
\usepackage{bm}

\usepackage[T1]{fontenc}
\usepackage{lmodern}
\usepackage{microtype}

\usepackage[a4paper, inner=3cm,outer=3cm,top=3.3cm,bottom=3cm]{geometry}

\usepackage{xspace}

\usepackage[autostyle]{csquotes}
\usepackage[dvipsnames]{xcolor}

\usepackage{tikz}
\usetikzlibrary{matrix,arrows,patterns,decorations.pathreplacing,svg.path}
\usepackage{pgfplots}

\usepackage{todonotes}

\newcommand{\h}{0.866025}

\usepackage[colorlinks, bookmarksdepth=3, citecolor=blue,linkcolor=blue]{hyperref}

\usepackage{pdfcomment}

\makeatletter
\@ifpackageloaded{hyperref}{%
\DeclareRobustCommand{\SkipTocEntry}[5]{}}{%
\DeclareRobustCommand{\SkipTocEntry}[4]{}}
\makeatother

\theoremstyle{plain}
\newtheorem{thm}{Theorem}[section]
\newtheorem{prop}[thm]{Proposition}
\newtheorem{cor}[thm]{Corollary}
\newtheorem{lem}[thm]{Lemma}
\newtheorem{conj}[thm]{Conjecture}
\newtheorem{prob}[thm]{Problem}
\newtheorem{defn}[thm]{Definition}

\theoremstyle{definition}
\newtheorem{ex}[thm]{Example}
\newtheorem{rem}[thm]{Remark}

\newcommand{\setcond}[2]{\left\{ #1 \,:\, #2 \right\}}

\DeclareMathOperator{\conv}{conv}

\newcommand{\N}{{\mathds{N}}}
\newcommand{\Q}{{\mathds{Q}}}
\newcommand{\R}{{\mathds{R}}}
\newcommand{\Z}{{\mathds{Z}}}

\newcommand{\Km}{{\mathcal{K}}}

\newcommand\blfootnote[1]{%
  \begingroup
  \renewcommand\thefootnote{}\footnote{#1}%
  \addtocounter{footnote}{-1}%
  \endgroup
}

\DeclarePairedDelimiter\floor{\lfloor}{\rfloor}

\DeclareMathOperator{\AF}{AFP}
\DeclareMathOperator{\AFC}{AFC}

\DeclareMathOperator{\MV}{V}
\DeclareMathOperator{\mv}{v}
\DeclareMathOperator{\mw}{w}
\DeclareMathOperator{\V}{V}
\DeclareMathOperator{\Vol}{Vol}

\title[Inequalities between mixed volumes of convex bodies: 
volume bounds]{Inequalities between mixed volumes of convex bodies: 
volume bounds for the Minkowski sum}
\author{Gennadiy~Averkov}
\address{Fakult\"at 1, BTU Cottbus-Senftenberg, Platz der Deutschen Einheit 1, 03046 Cottbus, Germany}
\email{averkov@b-tu.de}
\author{Christopher~Borger}
\address{Fakult\"at f\"ur Mathematik, Otto-von-Guericke-Universit\"at Magdeburg, Universit\"atsplatz 2, 39106 Magdeburg, Germany}
\email{christopher.borger@ovgu.de}
\author{Ivan~Soprunov}
\address{Department of Mathematics and Statistics, Cleveland State University,  2121 Euclid Ave, Cleveland, Ohio, 44115 USA}
\email{i.soprunov@csuohio.edu}

\begin{document}
\selectlanguage{english}

\maketitle

\begin{abstract}
In the course of classifying generic sparse polynomial systems which
are solvable in radicals, Esterov recently 
showed that the volume of the
Minkowski sum $P_1+\dots+P_d$ of $d$-dimensional lattice polytopes is bounded from
above by a function of order $O(m^{2^d})$, 
where $m$ is the mixed volume of the tuple
$(P_1,\dots,P_d)$. This is a consequence of the well-known
Aleksandrov-Fenchel inequality. Esterov also posed the
problem of determining a sharper bound. We show how additional
relations between mixed volumes can be employed to improve the 
bound to $O(m^d)$, which is asymptotically sharp.
We furthermore prove a sharp exact upper bound in dimensions $2$ and
$3$. Our results generalize to tuples of arbitrary convex bodies
with volume at least one.
\end{abstract}


\section{Introduction}

\subsection{Combinatorial structure of systems of algebraic equations with $m$ solutions}

Consider Laurent polynomials 
\[
	f_1,\ldots,f_d \in \mathbb{C}[z_1^{\pm 1}\,\ldots,z_d^{\pm 1}]
\] with  fixed Newton polytopes $P_1,\ldots,P_d$ and generic coefficients.
By the famous \emph{{Bernstein--Khovanskii--Kouchnirenko (BKK) theorem}} \cite{Bernstein75} (see also \cite[Section 7.5]{CLO05}), the number of solutions to the corresponding generic system of equations 
\begin{equation} \label{eq:l:system}
	f_1 = \cdots = f_d = 0
\end{equation}
in the complex torus $(\mathbb{C} \setminus \{0\})^d$ depends only on the tuple of Newton polytopes $P=(P_1,\ldots,P_d)$ and is equal to $\MV(P_1,\ldots,P_d)$, the so-called normalized mixed volume of $P_1,\ldots,P_d$.  This means that the number of solutions $m$ of a generic system \eqref{eq:l:system} can be computed purely combinatorially. It is an interesting task to revert this process and be able to infer structural and/or quantitative properties of the tuples $P$ associated to systems with a given number $m \in \Z_{> 0}$ of solutions. Recently, a number of results in this direction have been obtained. In \cite{Esterov2015} systems with exactly one solution have been completely classified by Esterov and Gusev. Esterov and Gusev \cite{Esterov2016} also classified systems with at most $4$ solutions in the case when all Newton polytopes $P_1,\ldots,P_d$ coincide up to translations. The authors of this manuscript classified, using an algorithmic approach, all systems with up to four solutions for $d=3$ and all systems with up to 10 solutions for $d=2$, see \cite{mvClass2019}. It is natural to expect a unifying  structure for all tuples $P= (P_1,\ldots,P_d)$ with small values of the mixed volume $m$. For example, such structural results were obtained in \cite{Esterov2015} for $m=1$ and conjectured by Esterov and Gusev for $m=2$ (private communication). However, as $m$ gets larger, we do not expect structural results for all tuples to hold, so it makes more sense to concentrate on the quantitative aspects, such as  volume bounds for the polytopes $P_i$ and their Minkowski sums. 

In this manuscript we focus on the case when all the Newton polytopes $P_1,\dots, P_d$ are full-dimensional. Esterov \cite{Esterov2019} has shown that in this case the volume of the Minkowski sum 
\[
	\Sigma(P):=P_1 + \cdots + P_d
\] 
has the asymptotic order at most $O(m^{2^d})$, as $m \to \infty$. This bound allows to control the sizes of the $P_i$ in the tuple $P$. In particular, it implies that the number of possible tuples $P$ of $d$-dimensional lattice polytopes with a given value of the mixed volume ${m}$ is finite, up to the natural equivalence consisting of permutation of the polytopes within the tuple, independent lattice translations of the polytopes, and a common unimodular transformation of all the polytopes of the tuple. In the course of showing this bound, Esterov \cite{Esterov2019}
also raised the question of determining a sharper bound for the volume of $\Sigma(P)$.

While our motivation comes from the theory of Newton polytopes, we do not exploit any combinatorial properties of lattice polytopes in this paper. In fact, our approach works in a more general context of convex bodies. However, we prefer to rescale the usual $d$-dimensional Euclidean volume by a factor of  $d!$, as it is common in the theory of Newton polytopes. We denote this normalized volume by~$\Vol$. We remark that any other rescaling of the Euclidean volume would work just as well. 

\subsection{Asymptotic behavior of $\Vol(\Sigma(P))$} We sketch the approach of Esterov. 
Consider a tuple $K= (K_1,\ldots,K_d)$ of convex bodies in $\R^d$ satisfying $\Vol(K_1) \ge 1,\ldots, \Vol(K_d) \ge 1$.
One can represent $\Vol(\Sigma(K))$ as the sum 
\begin{align}
	\label{vol:sum:via:mv}
	\Vol(\Sigma(K)) = \sum_{i_1,\ldots,i_d \in \{1,\ldots,d\}}\MV(K_{i_1},\ldots,K_{i_d}) 
\end{align} of all possible mixed volumes that can be built from 
 convex bodies $K_1,\ldots,K_d$ and then relate the mixed volumes $\MV(K_{i_1},\ldots,K_{i_d})$ via the \emph{Aleksandrov-Fenchel inequality}  \cite[Theorem~7.3.1]{Schneider2014}
\begin{equation}\label{AFI}\tag{AF}
	\MV(A,B,C)^2 \ge \MV(A,A,C) \MV(B,B,C),\quad\text{where }\ C=(C_3,\dots,C_d),
\end{equation}
which holds for any  convex bodies $A, B, C_3,\ldots,C_d$ in $\R^d$. Considering the system 
\begin{align}
	\label{AF:system}
	\V_{i,j,k_3,\ldots,k_{d}}^2  & \ge \V_{i,i,k_3,\ldots,k_{d}} \V_{j,j,k_3,\ldots,k_{d}},\quad \V_{k_1,\ldots,k_d}\ge 1, & & \forall \ i,j, k_1,\ldots,k_d \in \{1,\ldots,d\}
\end{align}
formally, as a system of inequalities in variables $\V_{i_1,\ldots,i_d}$, and using the condition $\V_{1,2,\ldots,d}=m$, Esterov deduced 
 a bound on each $\V_{i_1,\ldots,i_d}$ in terms of $m$ and $d$ by combining the inequalities of the system \eqref{AF:system}.
This produced the asymptotic estimate 
\begin{align}
\label{V:sum:Esterov}
\Vol(\Sigma(K))  \leq O(m^{2^d}),\quad \text{as } \ m \to \infty.
\end{align}
We call this approach \emph{the black-box application of \eqref{AFI}}.  %

The bound \eqref{V:sum:Esterov} shows that there exists an estimate of the form $\Vol(\Sigma(K)) = O(m^{\epsilon(d)})$, with $\epsilon(d) \le 2^d$. It is easy to see that $\epsilon(d)$ is at least $d$, since for $K=(mK_d,K_d,\dots, K_d)$ 
with $\Vol(K_d) = 1$, 
one has $\Vol(\Sigma(K))= (m+ d-1)^d$. We have been able to verify that, by the black-box application of \eqref{AFI}, the best exponent $\epsilon(d)$ that one can get satisfies $3^{(d-2)/3} \le \epsilon(d) \le 3^{d/3}$, which is surprisingly far from $d$, for large $d$. A priori, there might be different reasons for this situation: $\epsilon(d)$ might be much larger  than $d$ or \eqref{AFI}, applied in the black-box style, is too weak in the context of the problem. In the beginning of this project it was hard for us to believe that the latter could be the case, because \eqref{AFI} are very general inequalities that directly imply and subsume many other inequalities related to volumes of convex bodies, with the \emph{isoperimetric} and \emph{Brunn--Minkowski inequalities} among the most prominent examples. 
Nevertheless, we have been able to prove the following result (see Theorem~\ref{thm:asymp_Kound} for an explicit bound).

\begin{thm}\label{T:main result}
	Among all  convex bodies $K_1,\ldots,K_d$ in $\R^d$ satisfying
	\begin{align*}
		\Vol(K_1) \ge 1, \ldots, \Vol(K_d) \ge 1,\quad\text{and }\ \V(K_1,\ldots,K_d) = m,
	\end{align*}
	the maximum of $\Vol(K_1 + \cdots + K_d)$ is of order $O(m^d)$, as $m \to \infty$. 
\end{thm}

Interpreting Theorem~\ref{T:main result} in terms of the BKK theorem allows to derive the following corollary for generic system of polynomial equations: 

\begin{cor} \label{cor:systems}
	Let $f_1,\ldots,f_d \in \mathbb{C}[z_1^{\pm 1}\,\ldots,z_d^{\pm 1}]$ be generic Laurent polynomials with fixed $d$-dimensional Newton polytopes and let $m$ be the number of solutions of the system $f_1= \cdots = f_d = 0$ in the complex torus $(\mathbb{C} \setminus \{0\})^d$. Then the product $f_1 \cdots f_d$ is a Laurent polynomial containing at most $O(m^d)$ monomials, as $m \to \infty$. 
\end{cor}
The assertion of Corollary~\ref{cor:systems} agrees nicely with the well-known B\'ezout theorem \cite[Section~8.7]{CLO15}. Indeed, if $f_1,\dots, f_d$ are generic polynomials of total degrees $k_1,\dots, k_d$ then $m=k_1\cdots k_d$ by B\'ezout's theorem. Also the number of monomials in each $f_i$
is of order $k_i^d$. Therefore, the number of monomials in $f_1 \cdots f_d$ is of order $m^d$.

Our proof of Theorem~\ref{T:main result} uses the inequality 
\begin{equation}
	\label{square}
	\tag{$\square$}
	\MV(A,A,D) \MV(B,C,D) \le 2 \MV(A,B,D) \MV(A,C,D),\quad\text{where }\ D = (D_3,\ldots,D_d),
\end{equation}
valid for any  convex bodies $A, B, C, D_3,\ldots,D_d$ in $\R^d$. This inequality explicitly appears in \cite[Lemma 5.1]{BGL} and is derived using the same argument as in the proof of \cite[Lemma 7.4.1]{Schneider2014}. Interestingly, \eqref{square} is derived from \eqref{AFI} algebraically, but not in a black-box style. We sketch the proof of \eqref{square} in Section~\ref{S:weak}. While the estimate $\Vol(\Sigma(K)) = O(m^d)$ is obtained via a  black-box application of \eqref{square}, 
the derivation itself is non-trivial. For the asymptotic bound to be obtained, each single $\V_{i_1,\dots,i_d}$ must be estimated in terms of $m = \V_{1,\ldots,d}$, possibly tightly. This estimation task can be linked to a linear optimization problem, since by taking the logarithms (e.g., to the base $2$) of \eqref{square} we obtain the linear inequality 
\begin{equation}
	\label{log:square}
	\tag{$\log \square$}
	\log \MV(A,A,D) + \log \MV(B,C,D) \le 1 + \log \MV(A,B,D) + \log \MV(A,C,D)
\end{equation}
in the logarithms of the mixed volumes. The duality theory of linear programming tells us that the best upper bound on the terms $\mv_{i_1,\ldots,i_d} := \log \V_{i_1,\ldots,i_d} \in \R_{\ge 0}$, which can be derived from \eqref{log:square} using the black-box approach can be verified by taking a non-negative linear combination of the inequalities
\begin{align}
	\label{square:system}
	\mv_{i,i,k_3,\ldots,k_d} +  \mv_{s,t,k_3,\ldots,k_d} \le 1 +  \mv_{i,s,k_3,\ldots,k_d} + \mv_{i,t,k_3,\ldots,k_d} & & \forall \ i,s,t,k_3,\ldots,k_d \in \{1,\ldots,d\},
\end{align}
that arise by plugging  $K_1,\ldots,K_d$ into \eqref{log:square} in all possible ways, and taking into account that $\mv_{i_1,\ldots,i_d} \in \R_{\ge 0}$ and $\mv_{1,\ldots,d}= \log m$. This task might appear straightforward, because one ``only'' needs to combine \eqref{square:system} in such a way that the best possible bound on $\mv_{i_1,\ldots,i_d}$ can be confirmed. 
In fixed small dimensions we actually employed this approach using 
CPLEX \cite{cplex2015v12} as a solver for the resulting linear program. 
However, these computations also revealed that the linear combinations
of inequalities of the form \eqref{log:square} that confirm the best
possible bound are very complicated and use a vast amount of inequalities.
Furthermore, compared to \eqref{AFI}, the more complicated structure of \eqref{log:square} having the sum
of two mixed volumes as an upper bound  destroys any attempt of 
simple successive application of \eqref{log:square} similar to the application of \eqref{AFI} in Esterov's approach.
Our way of handling this complexity is to derive simpler inequalities 
with a single term on both sides (Theorem \ref{thm:asymp_Kound}).
We then show how
these inequalities can be successively applied to obtain the bound of order $O(m^d)$. 

\subsection{Exact bounds in small dimensions.} 
Esterov's approach gives an exact upper bound on $\Vol(\Sigma(K))$ in dimension two. It is not hard to check by directly applying \eqref{AFI}  with $d=2$ 
that the following holds.

\begin{prop}\label{two:maxima:in:dimension:two}
	Let $m \in \R_{\ge 1}$. Consider $2$-dimensional  convex bodies $K_1,K_2$ in $\R^2$ satisfying 
	\begin{align*}
		\Vol(K_1) \ge 1,\  \Vol(K_2) \ge 1,\quad\text{and }\ \MV(K_1,K_2) = m.
	\end{align*}
	Among all such bodies,
	\begin{enumerate}
		\item the maximum of $\Vol(K_1)$ is $m^2$ and
		\item the maximum of $\Vol(K_1 + K_2)$ is $(m+1)^2$.
	\end{enumerate}
	Both maxima are attained when $K_1 = m K_2$ and $\Vol(K_2) = 1$. 
\end{prop}

The inequality \eqref{AFI} is still strong enough to obtain the exact bound on $\Vol(\Sigma(K))$ in dimension three. We have the following result.

\begin{thm}
	\label{three:maxima:in:dimension:three}
	Let $m \in \R_{\ge 1}$. Consider $3$-dimensional  convex bodies $K_1,K_2,K_3 \subset \R^3$ satisfying 
	\begin{align*}
	\Vol(K_1)\ge 1,\ \Vol(K_2) \ge 1, \ \Vol(K_3) \ge 1,\quad\text{and } \ \MV(K_1,K_2,K_3)=m.
	\end{align*}
	Among all such bodies, 
	\begin{enumerate}
		\item \label{dim:3:1} the maximum of $\Vol(K_1)$ is $m^3$,
		\item \label{dim:3:2} the maximum of $\Vol(K_1 + K_2)$ is $(m+1)^3$, and
		\item \label{dim:3:3} the maximum of $\Vol(K_1 + K_2 + K_3)$ is $(m+2)^3$.
	\end{enumerate}
	All three maxima are attained when  $K_1 = m K_2 = m K_3$ and $\Vol(K_3) = 1$. 
\end{thm}

Theorem~\ref{three:maxima:in:dimension:three} is obtained using a computer assisted approach, which involved the linearization of \eqref{AFI}, produced in the same way as the linearization \eqref{log:square} of \eqref{square}  above. 

Based on the above evidence and the asymptotic behavior of $\Vol(\Sigma(P))$ presented in Theorem~\ref{T:main result} we propose the following conjecture.

\begin{conj}\label{conj}
Among all  convex bodies $K_1,\ldots,K_d$ in $\R^d$ satisfying
	\begin{align*}
		\Vol(K_1) \ge 1, \ldots, \Vol(K_d) \ge 1,\quad\text{and }\ \V(K_1,\ldots,K_d) = m,
	\end{align*}
for any $1\leq\ell\leq d$, 
the maximum of $\Vol(K_1 + \cdots + K_\ell)$ equals $(m+\ell-1)^d$ and is attained when $K_1 = mK_2 = \cdots = mK_d$ with $\Vol(K_d) = 1$.
\end{conj}

We know that the conjecture is true for $\ell=1$ (Remark~\ref{R:single-volume-bound}) and for $d \le 3$ (Proposition~\ref{two:maxima:in:dimension:two} and Theorem~\ref{three:maxima:in:dimension:three}). All other cases are open.

\subsection{Organization of the paper}

In Section~\ref{sec:prelim} we introduce basic preliminary definitions and results and fix notation. 
Section~\ref{sec:asymptotics} is devoted to employing the Aleksandrov-Fenchel inequalities to prove an upper bound 
for the volume of the Minkowski sum of a tuple of fixed mixed volume in general dimension.
Furthermore, we show that the relations providing this bound are best possible if we 
consider only the Aleksandrov-Fenchel inequalities. In Section~\ref{S:weak} we use additional 
inequalities between mixed volumes in order to prove a stronger bound
 on the volume of the
Minkowski sum that is asymptotically sharp.
In Section~\ref{sec:dim:3} we shift our attention from asymptotic bounds in 
general dimension to proving the exact bound in dimension $3$.
Finally, Section~\ref{sec:remarks} contains a result that allows to simplify Conjecture~\ref{conj} in the case $\ell < d$ and discusses open questions about the relations between mixed volumes for compact convex sets which  were motivated by the work on this project.

\subsection*{Acknowledgements} We are grateful to
Tobias Boege for helpful discussions regarding the combinatorial step of 
the proof of Theorem~\ref{thm:asymp_Kound}. The two first authors
and a research visit of the third author
were funded by the Deutsche Forschungsgemeinschaft (DFG, German Research Foundation) - 314838170, GRK 2297 MathCoRe.

\section{Preliminaries} \label{sec:prelim}

For $d \in \Z_{>0}$, let $[d]:=\{1,\ldots,d\}$. We use $e_1,\ldots,e_d$ to denote the standard basis vectors in $\R^d$. 
Consider a convex body $K\subset\R^d$, i.e. a compact convex set with positive $d$-dimensional volume. 
Let $\Vol(K)$ denote the normalized volume of $K$, that is  the usual Euclidean volume rescaled by a factor of $d!$. In particular, the normalized volume of the standard simplex $\conv\{0,e_1,\dots,e_d\}$  equals $1$. Recall that the {\it Minkowski sum} of two sets $A,B$ in $\R^d$ is the vector sum $$A+B=\{a+b : a\in A, b\in B\}.$$ Let $K_1,\dots, K_d$ be compact convex sets in $\R^d$. 
The (normalized) {\it mixed volume} $\MV(K_1,\dots,K_d)$
is the unique function in $K_1,\dots, K_d$ which is symmetric, multilinear with respect to Minkowski addition, and which coincides with the normalized volume on the diagonal, i.e. 
$$\MV(K,\dots, K)=\Vol(K)$$
for any compact convex set $K\subset \R^d$. Here is an explicit expression for the mixed volume \cite[Section 5.1]{Schneider2014}:
$$\MV(K_1,\dots,K_d)=\frac{1}{d!}\sum_{n=1}^d(-1)^{d+n}\!\sum_{i_1<\dots<i_n}\Vol(K_{i_1}+\dots+K_{i_n}).$$

Let us denote by $\Km_d$ the set of all convex bodies in $\R^d$ and let $\Km_{d,1} \subset \Km_d$ denote the subset of those 
convex bodies, whose normalized volume is at least $1$.

Fix $n > 0$  and a family $\Km$ of compact convex sets in $\R^d$, and consider an ordered $n$-tuple $K=(K_1,\dots, K_n)$ of elements in $\Km$. It defines a collection of $n^d$  mixed volumes
\[
\Big( \MV(K_{i_1},\ldots,K_{i_d}) : i_1,\dots, i_d\in\{1,\dots,n\}\Big).
\]
Since the mixed volume $\MV(K_{i_1},\ldots,K_{i_d})$ is invariant under permutation of the indices, we introduce an alternative notation 
\[
\MV_K(p_1,\ldots,p_n) = \MV( \underbrace{K_1,\ldots,K_1}_{p_1},\ldots,\underbrace{K_n,\ldots,K_n}_{p_n}).
\] 
In this notation, the \emph{mixed volume configuration} of an $n$-tuple $K=(K_1,\dots, K_n)$ is the vector 
\begin{equation}\label{e:m-config}
\Bigl (\V_K(p) \Bigr)_{p \in \Delta_{n,d}} \in \R_{\ge 0}^{\Delta_{n,d}}
\end{equation}
indexed by the set
\begin{equation}\label{e:simplex}
\Delta_{n,d} := \setcond{ (x_1,\ldots,x_n) \in \Z_{\ge 0}^n}{x_1 + \cdots + x_n = d}.
\end{equation}
For example, in the case $d=3$, $n=2$, the mixed volume configuration of a pair $K=(K_1,K_2)$ of $3$-dimensional convex bodies consists of the following four mixed volumes:
\begin{align*}
\V_K(3,0) & =  \MV(K_1,K_1,K_1) = \Vol(K_1), 
\\ \V_K(2,1) & =  \MV(K_1,K_1,K_2), 
\\ \V_K(1,2) & = \MV(K_1,K_2,K_2), 
\\ \V_K(0,3) & = \MV(K_2,K_2,K_2)  = \Vol(K_2).
\end{align*}

Furthermore, given a family of compact convex sets $\Km$, 
we define the {\it mixed volume configuration space}
\begin{equation}\label{e:m-config-space}
\MV(\Km,\Delta_{n,d}) := \setcond{ \Bigl(\V_K(p) \Bigr)_{p \in \Delta_{n,d}} }{K \in \Km^n},
\end{equation}
which represents all possible sets of values of the
different mixed volumes indexed by $p \in \Delta_{n,d}$ built for convex sets 
from $\Km$. When all sets from $\Km$ are full-dimensional, we also introduce the \emph{logarithmic mixed volume configuration space} 
\begin{equation}\label{e:m-log-config-space}
\mv(\Km,\Delta_{n,d}) := \setcond{ \Bigl (\mv_K(p) \Bigr)_{p \in \Delta_{n,d}} }{K \in \Km^n},\quad \text{where }\
\mv_K(p) := \log \V_K(p).
\end{equation}
Here and throughout the paper $\log$ denotes the logarithm to base 2.

Recall that, given an $n$-tuple $K = (K_1,\dots,K_n) \in (\Km_d)^n$, we
denote by $\Sigma(K) \coloneqq K_1 + \dots + K_n \in \Km_d$ the {Minkowski sum} over its elements and by $\V_K\in \R^{\Delta_{n,d}}$
the mixed volume configuration corresponding to $K$, see (\ref{e:m-config}).

In what follows, we have two points of view for the elements of $\R^{\Delta_{n,d}}$. On  one hand, $\R^{\Delta_{n,d}}$ is a vector space over $\R$ and we can treat its elements merely as vectors of $\R^N$ with $N = |\Delta_{n,d}|$. On the other hand, 
since $\Delta_{n,d}$ is a subset of $\R^d$, we can talk about elements 
of $\R^{\Delta_{n,d}}$ as functions on $\Delta_{n,d}$ which may or may not possess some discrete concavity properties. 
Because of this, we will call the elements of $\R^{\Delta_{n,d}}$ functions or vectors depending on the context.

The following proposition, which immediately follows from the multilinearity of the mixed volume, relates 
the volume of the Minkowski sum $\Sigma(K)$ to the mixed volume configuration of $K$.

\begin{prop}
\label{prop:formula_mink}
Let $K \in (\Km_d)^n$. Then we have the following formula for the
volume of the Minkowski sum of the elements in $K$:
\begin{align} \label{vol:sigma:eq}
\Vol(\Sigma(K)) = \sum_{p \in \Delta_{n,d}} 
\binom{d}{p} \MV_K(p),
\end{align}   
where $\binom{d}{p}=\frac{d!}{p_1!\cdots p_n!}$ denotes the multinomial coefficient for $p=(p_1,\dots,p_n)$.
\end{prop}

Qualitatively, Proposition~\ref{prop:formula_mink} shows that $\Vol(\Sigma(K))$ is a linear function of $\MV_K \in \R^{\Delta_{n,d}}$. Since the convex bodies in the tuple $K$ have positive volume, it implies that $\MV_K \in \R_{>0}^{\Delta_{n,d}}$ (see, for example, Lemma~\ref{lem:concavity_simplices} below). 
Hence, we can use the logarithmic mixed volume configuration $\mv_K \in \R^{\Delta_{n,d}}$ and reformulate \eqref{vol:sigma:eq} as 
\[
	\Vol(\Sigma(K)) = \sum_{p \in \Delta_{n,d}}\binom{d}{p}  2^{\mv_K(p)}.
\]
This shows that $\Vol(\Sigma(K))$ is a convex function of $\mv_K$. When the convex bodies in the tuple $K$ have volume
at least 1, Lemma~\ref{lem:concavity_simplices} implies that $\MV_K \in \R_{\geq 1}^{\Delta_{n,d}}$ and, consequently, $\mv_K \in \R_{\geq 0}^{\Delta_{n,d}}$.

Below we restate the Aleksandrov-Fenchel inequalities \eqref{AFI} in
the $\MV_K(p)$-notation introduced in Section~\ref{S:relations}.

\begin{thm}[Aleksandrov-Fenchel Inequalities]
\label{thm:af}
Let $i,j \in [d]$ with $i \neq j$ and $p = (p_1,\ldots,p_d) \in \Delta_{d,d}$ a point satisfying 
$p_i,p_j \geq 1$. Then, for every $d$-tuple $K$ of $d$-dimensional convex bodies in $\R^d$, one has 
\begin{align}\label{e:AF}
\MV_K(p)^2 \geq \MV_K(p+e_i-e_j) \MV_K(p-e_i+e_j).
\end{align}
Equivalently, in the $\log$-notation, one has 
\begin{align}
\label{e:af}
\tag{$\log$ AF}
2\mv_K(p) \geq \mv_K(p+e_i-e_j)+\mv_K(p-e_i+e_j).
\end{align}
\end{thm} 

Recall that a sequence $r_0,r_1\dots, r_n$ of non-negative real numbers is {\it log-concave} if 
$r_i^2 \geq r_{i-1}r_{i+1}$ holds for all $0 < i < n$. Furthermore, a sequence $r_0,\ldots,r_n$ of arbitrary real numbers is called {\it concave} if $2r_i \geq r_{i-1}+r_{i+1}$ for all $0 < i < n$.
In this terminology, \eqref{e:AF} is the discrete log-concavity property of the function $\MV_K\in \R^{\Delta_{d,d}}$ along the direction
$e_i-e_j$ for every $i,j \in [d]$ and $i \ne j$. Equivalently, \eqref{e:af} describes the concavity of  $\mv_K \in \R^{\Delta_{d,d}}$  in the direction $e_i -e_j$. See also Fig.~\ref{fig:config:3:3} for an illustration in the case $d=3$.

\begin{figure}
\begin{center}
	\begin{tikzpicture}[scale=1.4]
	\draw[red,very thick] (0.0,0.0) -- (3.0,5.196152422706632);
	\draw[red, very thick] (0.0,0.0) -- (6.0,0.0);
	\draw (0.0,0.0) -- (0.0,0.0);
	\draw[cyan, very thick] (2.0,0.0) -- (4.0,3.464101615137755);
	\draw[cyan, very thick] (1.0,1.7320508075688774) -- (5.0,1.7320508075688774);
	\draw (2.0,0.0) -- (1.0,1.7320508075688774);
	\draw (4.0,0.0) -- (5.0,1.7320508075688774);
	\draw (2.0,3.464101615137755) -- (4.0,3.464101615137755);
	\draw[cyan, very thick] (4.0,0.0) -- (2.0,3.464101615137755);
	\draw (6.0,0.0) -- (6.0,0.0);
	\draw (3.0,5.196152422706632) -- (3.0,5.196152422706632);
	\draw[red,very thick] (6.0,0.0) -- (3.0,5.196152422706632);
	\filldraw[fill=white] (0.0,0.0) circle (0.1);
	\node[above left] (0-0-3) at (0.0,0.0)  {\small $(0,0,3)\;\;$};
	\filldraw[fill=white] (1.0,1.7320508075688774) circle (0.1);
	\node[above left] (0-1-2) at (1.0,1.7320508075688774)  {\small $(0,1,2)\;\;$};
	\filldraw[fill=white] (2.0,3.464101615137755) circle (0.1);
	\node[above left] (0-2-1) at (2.0,3.464101615137755)  {\small $(0,2,1)\;\;$};
	\filldraw[fill=white] (3.0,5.196152422706632) circle (0.1);
	\node[above left] (0-3-0) at (3.0,5.196152422706632)  {\small $(0,3,0)\;\;$};
	\filldraw[fill=white] (2.0,0.0) circle (0.1);
	\node[above left] (1-0-2) at (2.0,0.0)  {\small $(1,0,2)\;\;$};
	\filldraw[fill=white] (3.0,1.7320508075688774) circle (0.1);
	\node[above left] (1-1-1) at (3.0,1.7320508075688774)  {\small $(1,1,1)\;\;$};
	\filldraw[fill=white] (4.0,3.464101615137755) circle (0.1);
	\node[above left] (1-2-0) at (4.0,3.464101615137755)  {\small $(1,2,0)\;\;$};
	\filldraw[fill=white] (4.0,0.0) circle (0.1);
	\node[above left] (2-0-1) at (4.0,0.0)  {\small $(2,0,1)\;\;$};
	\filldraw[fill=white] (5.0,1.7320508075688774) circle (0.1);
	\node[above left] (2-1-0) at (5.0,1.7320508075688774)  {\small $(2,1,0)\;\;$};
	\filldraw[fill=white] (6.0,0.0) circle (0.1);
	\node[above left] (3-0-0) at (6.0,0.0)  {\small $(3,0,0)\;\;$};
	\end{tikzpicture}
\end{center}
\caption{\label{fig:config:3:3}
Illustration of \eqref{e:af} for $d=3$. Elements of $\R^{\Delta_{3,3}}$ are real-valued functions on the ten lattice points of the triangle with the vertices $3 e_1, 3 e_2, 3 e_3$. The restriction of $\mv_K \in \R^{\Delta_{3,3}}$ to any of the red and any of the cyan segments generates a concave sequence.}
\end{figure}
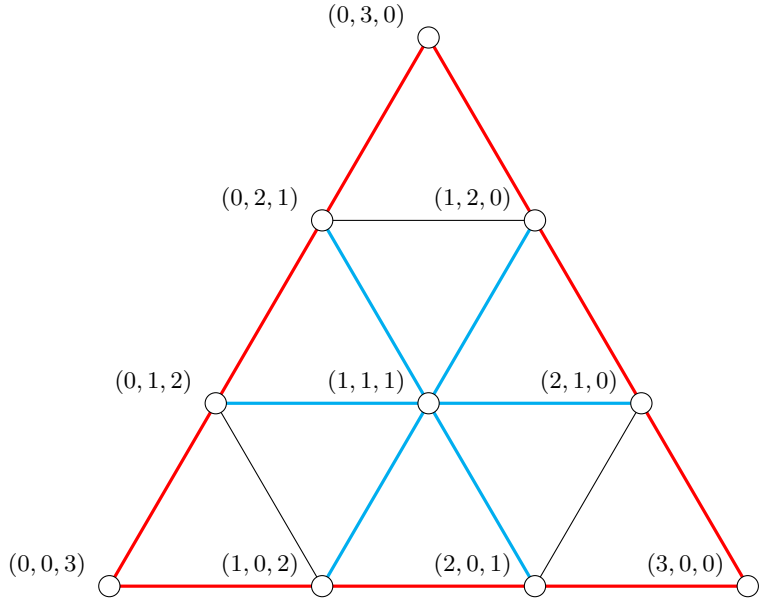

Concave and log-concave sequences are well studied in convex analysis and combinatorics. In Section~\ref{S:weak} we 
will work with relations of the more general type
\begin{align} \label{eq:weak:conc}
	\frac{1}{2} r_{i-1}+ \frac{1}{2} r_{i+1}\leq  r_i + C
\end{align}
that depend on a constant $C \ge 0$. We informally refer to 
inequalities of the form \eqref{eq:weak:conc} as \emph{weak concavity relations}. 
In the following lemma we include basic properties of sequences satisfying 
such weak concavity relations, 
which mimic basic properties of concave sequences.
 
\begin{lem}
\label{lem:concavity_along_line}
Let $r_0,r_1\dots,r_n$ be a sequence of non-negative real numbers satisfying \eqref{eq:weak:conc} for all $0 < i < n$ for some constant $C \ge 0$. Then 
\begin{align*}
\text{\rm (i) }\ & \frac{1}{2} r_{i-1}+ \frac{1}{2}r_{j+1}\leq 
\frac{1}{2} r_i+ \frac{1}{2} r_j + (j-i+1)C\ \text{ for all }\ 0< i\leq j<n,\\
\text{\rm (ii) }\ & \frac{n-1}{n} r_0+ \frac{1}{n} r_n\leq r_1 + ({n-1})C,\\
\text{\rm (iii) }\ & \frac{n-k}{n} r_0+ \frac{k}{n} r_n\leq r_k + k(n-k)C\ \text{ for all }\ 1\leq k\leq n.
\end{align*}
\end{lem}

\begin{proof}

(i) This follows by adding (and simplifying) the inequalities
$\frac{1}{2} r_{k-1}+\frac{1}{2}r_{k+1}\leq r_k + C$ for $i \leq k \leq j$.

\smallskip

\noindent(ii) For every $0< i< n$ we have $\frac{n-i}{2}r_{i-1}+
\frac{n-i}{2} r_{i+1} \leq (n-i)r_i + (n-i)C$. Adding these inequalities and simplifying we obtain the required inequality. 

\smallskip

\noindent(iii) We use induction on $k$. For $k=1$ this is the statement of part (ii). Assume 
$$\frac{n-k}{n} r_0+ \frac{k}{n} r_n\leq r_k + k(n-k)C.$$ 
Applying this to the sequence ${r_1,\dots, r_n}$ we get
\begin{equation}\label{e:first}
\frac{n-k-1}{n-1} r_1+ \frac{k}{n-1} r_n\leq r_{k+1} + k(n-1-k)C.
\end{equation}
Applying part (ii) to the sequence ${r_k,\dots, r_n}$ we get
\begin{equation}\label{e:second}
\frac{n-k-1}{n-k}r_k+ \frac{1}{n-k} r_n\leq r_{k+1} + (n-k-1)C.
\end{equation}
From part (i) we have
\begin{equation}\label{e:third}
r_0+r_{k+1}\leq r_1+r_{k} + 2kC.
\end{equation}
Finally, multiplying \eqref{e:first} by $n-1$, \eqref{e:second} by $n-k$, and 
\eqref{e:third} by $n-k-1$ and adding the results
we obtain $$(n-k-1)r_0+(k+1)r_n\leq nr_{k+1}+n(k+1)(n-k-1)C,$$ as required.
\end{proof}

\begin{rem}\label{r:restate} 
It is convenient to restate (iii) in Lemma \ref{lem:concavity_along_line} in a more symmetric form:
\begin{equation}\label{e:restate-iii}
\frac{k}{k+l} r_{p-l}+ \frac{l}{k+l} r_{p+k}\leq r_p + klC,
\end{equation}
for any $0\leq l \leq p$ and $0\leq k\leq n-p$.
\end{rem}

\section{The asymptotics derived from the Aleksandrov-Fenchel inequalities} \label{sec:asymptotics}

The goal of this section is to investigate the relations among mixed volumes
that follow from the Aleksandrov-Fenchel inequalities in the black box fashion
and to study the sharpness of such relations. 

\subsection{Relations and bounds coming from Aleksandrov-Fenchel inequalities}

%
%


The following lemma shows how Aleksandrov-Fenchel inequalities yield certain
higher-order $\log$-concavity relations on the function 
$\MV_K \in \R^{\Delta_{d,d}}$.

 \begin{lem}[{Concavity Relations from Aleksandrov-Fenchel}]
 \label{lem:concavity_simplices}
 	For $n, k \in [d]$, consider a ``copy'' of $\Delta_{n,k}$ in $\Delta_{d,d}$ given by 
 	\[
 		S = \setcond{c_1 e_{i_1} + \cdots + c_n e_{i_n} + t }{
 		(c_1,\ldots,c_n) \in \Delta_{n,k} },
 	\] 
 	where $1 \le i_1 < \cdots < i_n \le d$ and 
 	$t \in \Z_{\geq 0}^d$ satisfies $t_1+\dots+t_d=d-k$. Denote
 	the vertices of $\conv(S)$ by $b_j = k e_{i_j}+t \in \Delta_{d,d}$ 
 	for $j \in [n]$ . 
 	Then, for every $K \in (\Km_d)^d$ and every $p \in S$, the mixed volume
 	configuration $\MV_K$ satisfies the log-concavity relation
 	\begin{align}
 		\label{eq:conc_rel}
 		\MV_K(p)^k \ge \MV_K(b_1)^{c_1} \cdots \MV_K(b_{n})^{c_n},
 	\end{align}
 	where $(c_1,\ldots,c_n) \in \Delta_{n,k}$ is the unique
 	vector satisfying $kp= c_1 b_1 + \dots + c_{n} b_{n}$. 
 \end{lem}
\begin{proof}
For the sake of readability we pass to proving an equivalent logarithmic
version of \eqref{eq:conc_rel}, that is, we show the inequality
\begin{align*}
\mv_K(p) \ge \frac{c_1}{k} \mv_K(b_1) + \dots + 
\frac{c_{n}}{k} \mv_K(b_n).
\end{align*}

We will prove the statement by induction on the number $n$ of vertices
of $S$. For $n=2$ the statement follows directly from Remark~\ref{r:restate} together with Theorem~\ref{thm:af}. Let $n$ now be an arbitrary positive 
integer and assume without loss of generality that the vertices of $S$ are
of the form $b_i=k e_i +t$ for all $i \in [n]$. 
We may assume that $p$ is an interior point of $\conv(S)$, as otherwise 
we can pass to the face of $\conv(S)$ containing $p$ and obtain the
statement by induction. 
It is straightforward to verify that the
line $p+\R (e_1-e_2)$ intersects the two facets 
$F_1=\conv(b_2,\dots,b_n)$ and $F_2=\conv(b_1,b_3,\dots,b_n)$ of 
$\conv(S)$ in  
lattice points $a_1,a_2$ in the relative interior of 
$F_1$ and $F_2$, respectively. Then $\mv_K(p)=\mv_K(a_1+ \tau (e_1-e_2))$ for
some $\tau \in \Z_{\geq 1}$ and, by Theorem~\ref{thm:af}, the 
logarithmic mixed volumes 
$$\mv_K(a_1),\mv_K(a_1+(e_1-e_2)),\dots,\mv_K(a_1+\tau (e_1-e_2)),\dots,\mv_K(a_2)$$ form a concave sequence. By Remark~\ref{r:restate}, this implies 
\begin{align}
\label{e:ind_step}
\mv_K(p) \geq \sigma_1 \mv_K(a_1) + \sigma_2 \mv_K(a_2),
\end{align}
for unique rational positive numbers 
$\sigma_1,\sigma_2 \in \Q_{> 0}$
with $\sigma_1+\sigma_2=1$ and $p=\sigma_1 a_1 + \sigma_2 a_2$. 
As $a_1$ and $a_2$ are lattice points in the relative interior of the
facets $F_1$ and $F_2$, respectively, one has
\begin{align*}
a_1=\mu^1_2 b_2 + \mu^1_3 b_3 + \dots+ \mu^1_{n} b_{n},
\hspace*{\baselineskip} a_2=\mu^2_1 b_1 + \mu^2_3 b_3 + \dots+ \mu^2_{n} b_{n},
\end{align*}
for some positive rational numbers $\mu^1_2,\mu^1_3,\dots,\mu^1_{n},
\mu^2_1,\mu^2_3,\dots,\mu^2_n \in \Q_{> 0}$.
By the induction hypothesis this implies
\begin{align*}
\mv_K(a_1)&\geq \mu^1_2 \mv_K(b_2)+\mu^1_3 \mv_K(b_3)+\dots+ \mu^1_{n} \mv_K(b_{n}),\\  
\mv_K(a_2)&\geq \mu^2_1 \mv_K(b_1)+\mu^2_3 \mv_K(b_3) + \dots+ 
\mu^2_{n} \mv_K(b_{n}).
\end{align*}
Combining this with \eqref{e:ind_step} one obtains
\begin{align*}
\mv_K(p) &\geq \left(\sigma_2 \mu^2_1 \right) \mv_K(b_1)
+ \left(\sigma_1 \mu^1_2 \right)
\mv_K(b_2)
+ \left( \sigma_1 \mu^1_3 + \sigma_2 \mu^2_3 \right) \mv_K(b_3)
+\dots + \left( \sigma_1 \mu^1_n + \sigma_2 \mu^2_n \right) \mv_K(b_n).
\end{align*}
By construction, the coefficients on the right hand-side 
satisfy 
\begin{align}
\label{eq:bar_coords}
p = \left(\sigma_2 \mu^2_1 \right) b_1
+ \left(\sigma_1 \mu^1_2 \right)
b_2+ \left( \sigma_1 \mu^1_3 + \sigma_2 \mu^2_3 \right) b_3
+\dots + \left( \sigma_1 \mu^1_n + \sigma_2 \mu^2_n \right) b_n,
\end{align}
which proves the claim as the barycentric coordinates of $p$ 
with respect to the vertices $b_1,\dots,b_n$ are unique
(in particular, all coefficients in \eqref{eq:bar_coords}
are integral multiples of $\frac{1}{k}$ by construction of $S$). 
\end{proof}

\begin{rem}\label{R:single-volume-bound}
The particular case of Lemma~\ref{lem:concavity_simplices} when $S=\Delta_{d,d}$ and $p={\bm 1}$ provides the following bound for the product of the volumes of the $K_i$:
$$\MV_K({\bm 1})^d\geq \Vol(K_1)\cdots\Vol(K_d).$$
This inequality can also be found in \cite[(7.64)]{Schneider2014}. In particular, we see that if all $K_i$ have volume at least 1
then  $\Vol(K_i) \le \MV_K({\bm 1})^d$ for every $i \in [d]$. 
\end{rem}

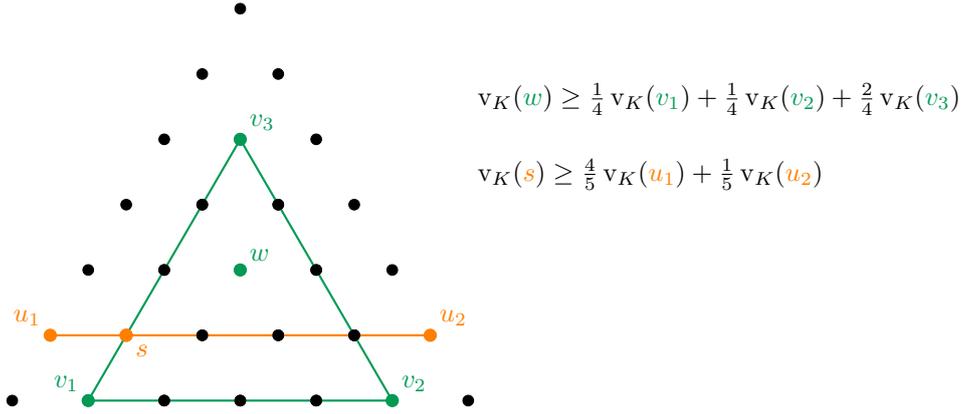
\begin{figure}
\begin{tikzpicture}[scale=1]
		\filldraw (0,0) circle (0.07);
		\filldraw (1,0) circle (0.07);
		\filldraw (2,0) circle (0.07);
		\filldraw (3,0) circle (0.07);
		\filldraw (0.5,\h) circle (0.07);
		\filldraw (1.5,\h) circle (0.07);
		\filldraw (2.5,\h) circle (0.07);
		\filldraw (1,2*\h) circle (0.07);
		\filldraw (2,2*\h) circle (0.07);
		\filldraw (1.5,3*\h) circle (0.07);
		\filldraw (4,0) circle (0.07);
		\filldraw (3.5,\h) circle (0.07);
		\filldraw (3,2*\h) circle (0.07);
		\filldraw (2.5,3*\h) circle (0.07);
		\filldraw (2,4*\h) circle (0.07);
		\filldraw (5,0) circle (0.07);
		\filldraw (4.5,\h) circle (0.07);
		\filldraw (4,2*\h) circle (0.07);
		\filldraw (3.5,3*\h) circle (0.07);
		\filldraw (3,4*\h) circle (0.07);
		\filldraw (2.5,5*\h) circle (0.07);
		\filldraw (6,0) circle (0.07);
		\filldraw (5.5,\h) circle (0.07);
		\filldraw (5,2*\h) circle (0.07);
		\filldraw (4.5,3*\h) circle (0.07);
		\filldraw (4,4*\h) circle (0.07);
		\filldraw (3.5,5*\h) circle (0.07);
		\filldraw (3,6*\h) circle (0.07);
		
		\filldraw[orange] (0.5,\h) node[above left]{$u_1$} circle (0.08);
		\filldraw[orange] (1.5,\h) node[below right]{$s$} circle (0.08);
		\filldraw[orange] (5.5,\h) node[above right]{$u_2$} circle (0.08);
		\draw[orange,thick] (0.5,\h)--(5.5,\h);
		\filldraw (2.5,\h) circle (0.07);
		\filldraw (3.5,\h) circle (0.07);
		\filldraw (4.5,\h) circle (0.07);
		
		\filldraw[ForestGreen] (1,0) node[above left]{$v_1$} circle (0.08);
		\filldraw[ForestGreen] (5,0) node[above right]{$v_2$} circle (0.08);
		\filldraw[ForestGreen] (3,4*\h) node[above right]{$v_3$} circle (0.08);
		\filldraw[ForestGreen] (3,2*\h) node[above right]{$w$} circle (0.08);
		\draw[ForestGreen,thick] (1,0)--(5,0);
		\draw[ForestGreen,thick] (5,0)--(3,4*\h);
		\draw[ForestGreen,thick] (3,4*\h)--(1,0);
		
		\filldraw[orange] (1.5,\h) circle (0.08);
		\filldraw (4.5,\h) circle (0.07);
		\filldraw (2,0) circle (0.07);
		\filldraw (3,0) circle (0.07);
		\filldraw (4,0) circle (0.07);
		\filldraw (4,2*\h) circle (0.07);
		\filldraw (3.5,3*\h) circle (0.07);
		\filldraw (2.5,3*\h) circle (0.07);
		\filldraw (2,2*\h) circle (0.07);
		
		\node[right] at (6,4)
		{$\mv_K({\color{ForestGreen}w}) \geq \frac{1}{4}\mv_K({\color{ForestGreen}v_1})+\frac{1}{4}\mv_K({\color{ForestGreen}v_2})+\frac{2}{4}\mv_K({\color{ForestGreen}v_3})$};
		\node[right] at (6,3)
		{$\mv_K({\color{orange}s}) \geq \frac{4}{5}\mv_K({\color{orange}u_1})+\frac{1}{5}\mv_K({\color{orange}u_2})$};
		\end{tikzpicture}
\caption{Two examples of concavity relations
of the type shown in Lemma~\ref{lem:concavity_simplices}.}
\end{figure}

The following is the main statement of this section which provides bounds 
that the Aleksandrov-Fenchel relations yield for the mixed volume
$\MV_K(p)$ for any $p \in \Delta_{d,d}$ when $\MV_K(\bm{1})$ is fixed. 

\begin{thm}[Bounds from Aleksandrov-Fenchel inequalities]
\label{thm:af_Kounds}
Let $K \in (\Km_{d,1})^d$ be a $d$-tuple of $d$-dimensional convex bodies of
volume at least $1$ and $p \in  \Delta_{d,d}$. 
Then 
\begin{align}
\label{eq:af_pbound}
\mv_K(p) \leq \mv_K(\bm{1})\prod_{i\, :\, p_i>0}p_i.
\end{align}
Furthermore, given that $\MV_K(\bm{1})=m$, one 
obtains the following bound: 
\begin{align}
\label{eq:af_minkbound}
	\Vol(\Sigma(K)) \le m^{3^{q}2^{r}}  d^d,
\end{align}
where $d=3q+2r$ with $q\in\Z_{\geq 0}$ and $r\in \{0,1,2\}$.
\end{thm}

\begin{proof}
We prove \eqref{eq:af_pbound} by inductively making use of 
Lemma~\ref{lem:concavity_simplices}. The induction is over the number
of zero entries of $p$ which we denote by $k$. 
Let us without loss of generality restrict to the case that
$p$ is decreasing, that is $p_1 \geq \dots \geq p_d$.

As $k=0$ implies $p=\bm{1}$ the statement is trivially fulfilled in this case.
Now let $k \in [d-1]$ be arbitrary.  Assume $p = (p_1,\dots,p_{d-k},0,\dots,0)$ has exactly $k$ zero entries and assume that the statement is true for any vector with at most $k-1$ zero entries.
Consider the vector
\begin{align*}
p' = (1,p_2,\dots,p_{d-k},\underbrace{1,\dots,1}_{p_1-1 \text{ times}},0,\dots,0).
\end{align*}
Clearly $p'$ has fewer zero entries than $p$ and, therefore,
\begin{align*}
\mv_K(p') \leq \mv_K(\bm{1})\prod_{i\, :\, p'_i>0}p'_i =\mv_K(\bm{1})\, p_2 \cdots p_{d-k}. 
\end{align*}
However, if one writes $p'$ as the barycenter of a suitable $(p_1-1)$-simplex,
Lemma~\ref{lem:concavity_simplices} yields
\begin{align*}
\mv_K(p') \geq &\frac{1}{p_1} \mv_K(p_1,p_2,\dots,p_{d-k},0,\dots,0) \\
+ &\frac{1}{p_1} \mv_K(0,p_2,\dots,p_{d-k},p_1,0,\dots,0) \\ 
+ &\frac{1}{p_1} \mv_K(0,p_2,\dots,p_{d-k},0,p_1,0,\dots,0) \\
+ & \cdots \\
+ &\frac{1}{p_1} \mv_K(0,p_2,\dots,p_{d-k},0,\dots,0,p_1,0,\dots,0).
\end{align*}
In particular, $\mv_K(p)=\mv_K(p_1,\dots,p_{d-k},0,\dots,0) \leq p_1 \mv_K(p')$, 
as we assumed the volumes of the $K_i$ to be at least $1$ and therefore
all terms on the right hand-side of the above inequality are non-negative.
This proves \eqref{eq:af_pbound}.

We now proceed to using \eqref{eq:af_pbound} in order to show the bound \eqref{eq:af_minkbound}. 
Write $d=3q+2r$ for unique non-negative integers $q,r$ with $r\in\{0,1,2\}$.
We first show that the maximal value
of $\prod_{i\, :\, p_i>0}p_i \eqqcolon g(p)$ is attained at a point 
$p_{\max}$ with $q$ entries equal to $3$, $r$ entries equal to $2$, and the remaining entries equal to $0$.

Note first that $2^{\floor{k/2}}>k$ for all $k \geq 6$. Therefore,
for any point $p \in \Delta_{d,d}$ with one coordinate being $k \geq 6$,
we can construct another point $p'$ by replacing the entry 
with value $k$ with $\floor{k/2}$ entries with value $2$ and obtain
$g(p')>g(p)$. Similarly, any entry with value $5$ in $p$ can be replaced
by two entries with values $2$ and $3$ respectively to increase the
value of $g$. As also any entry with value $4$ can be replaced by two
entries both with value $2$ without changing the value of $g$, this
shows that there exists a point $p$ maximizing $g$ with $p_i \leq 3$
for all $i \in [d]$. If $p$ has an entry with value $1$, one 
can construct a point increasing the value of $g$ by replacing 
$1,3$ with $2,2$, or $2,1$ with $3$, or $1,1$ with $2$.
One of these replacements is always possible and, hence,  a point $p$
maximizing $g$ can be chosen such that $p_i \in \{0,2,3\}$ for
all $i \in [d]$. Finally the observation that $2 \cdot 2 \cdot 2
< 3 \cdot 3$ shows that the maximum of $g$ is actually attained
by $p_{\max}$.

Combining this insight with Proposition~\ref{prop:formula_mink}
one obtains that, for any tuple $K \in (\Km_{d,1})^d$, one has
\begin{align*}
\Vol(\Sigma(K)) = \sum_{p \in \Delta_{d,d}} 
\binom{d}{p} 2^{\mv_K(p)} \leq 
\sum_{p \in \Delta_{d,d}}\binom{d}{p} 2^{\mv_K(\bm{1})g(p_{\max})}  = d^d m^{g(p_{\max})},
\end{align*}
where $m =2^{\mv_K(\bm{1})}=\MV_K(\bm{1})$. This shows
\eqref{eq:af_minkbound}.
\end{proof}

\subsection{On the optimality of Theorem~\ref{thm:af_Kounds}}

This subsection is devoted to showing that Theorem~\ref{thm:af_Kounds} 
actually provides the best bounds that one can get by only using Aleksandrov-Fenchel inequalities in what we call 
black-box style in the introduction. In order to make this 
term precise we need to define the set of all positive-valued functions on $\Delta_{d,d}$ that satisfy all 
linearized Aleksandrov-Fenchel inequalities \eqref{e:af}. In this
language, a statement that is obtained in black-box style from
the Aleksandrov-Fenchel inequalities is a statement that holds 
for each function in this set.

\begin{defn}
We define the {\it Aleksandrov-Fenchel
cone} $\AFC_d \subset \R^{\Delta_{d,d}}$ as the set of all $\mv \in \R^{\Delta_{d,d}}$ satisfying
\begin{align*}
	\mv(p)& \ge 0 & & \text{for all} \ p \in \{d e_1,\ldots, d e_d\}, \ \text{and}
	\\ 2 \mv(p) & \ge \mv(p + e_i - e_j) + \mv(p -e_i + e_j) & & \text{for all} \ p, p \pm (e_i-e_j) \in \Delta_{d,d} \ \text{with} \ i,j \in [d].
\end{align*}
We also define the \emph{Aleksandrov-Fenchel polytope} $\AF_d$ to be the following hyperplane section of $\AFC_d$:
\[
	\AF_d := \setcond{\mv \in \AFC_d}{\mv(\mathbf{1}) = 1}.
\]
\end{defn}
The Aleksandrov-Fenchel inequality implies
\[
	\mv(\Km_{d,1},\Delta_{d,d}) \subseteq \AFC_d.
\]
Furthermore, for all $d$-tuples $K\in (\Km_{d,1})^d$ with $\MV_K(\mathbf{1}) = m$, we have $\mv_K(\mathbf{1}) = \log m$ and, hence, $\mv_K \in (\log m) \AF_d$. 

\begin{rem}\label{rem:AFP}
It is straightforward to verify that Theorem~\ref{thm:af_Kounds} and
in particular Lemma~\ref{lem:concavity_simplices} are proven by 
iterated linear combination of inequalities \eqref{e:af}.  
This shows that the bound (\ref{eq:af_pbound})  in Theorem~\ref{thm:af_Kounds} holds for any $\mv\in(\log m) \AF_d$. 
\end{rem}

The following proposition shows that Theorem~\ref{thm:af_Kounds} provides 
the best possible bounds that can be deduced from Aleksandrov-Fenchel
inequalities in a black-box style.

\begin{prop}
\label{prop:AF_tight}
	Let $p^{\ast} \in \Delta_{d,d}$. Then 
	\[
		\max_{w \in \AF_d} w(p^{\ast}) =\prod_{i\, :\, p^{\ast}_i>0}p^{\ast}_i.
	\]
\end{prop}
\begin{proof}
	Let $p^\ast=(p_1^\ast,\ldots,p_d^\ast)$.
	Without loss of generality, we can assume that the entries of $p$ are sorted in descending order. Let $r \in [d]$ be the largest number 
	satisfying $p^{\ast}_r > 0$.
	
	The fact that $\prod_{i\, :\, p_i>0}p_i = p_1 \cdots p_r$ is an upper bound is
	true by Theorem~\ref{thm:af_Kounds} (see Remark~\ref{rem:AFP}). It remains to confirm that this value is indeed the maximum. To this end, consider $w \in \R^{\Delta_{d,d}}$ given by
	\begin{align*}
		w(p) =  p_1 \cdots p_r \quad \text{for} \ p \in \Delta_{d,d}.
	\end{align*}
	Under this assumption, we see that for the chosen $w$ one has $w(p^\ast) = \prod_{i\, :\, p^{\ast}_i>0}p^{\ast}_i$. 
	
	It remains to verify that $w \in \AF_d$. 
	We need to show that $w \in \R^{\Delta_{d,d}}$ is discretely concave in the directions $e_i - e_j$ with $i \ne j$ in the variables $p=(p_1,\ldots,p_d)$. The function $w$ is a product of some of the variables $p_1,\ldots,p_d$. If neither $p_i$ nor $p_j$ occurs in the product, $w(p)$ is constant therefore concave
	in direction $e_i-e_j$. If exactly one of the variables $p_i$ and $p_j$ occurs in the product, then the function is linear in direction $e_i-e_j$. Now consider the case that both $p_i$ and $p_j$ occur in the product. For simplicity, let $i=1$ and $j=2$, so

	\[
		w(p) = p_1 p_2 u,
	\] 
	where $u = \prod_{i=3}^{r}p_i \ge 0$  is independent of $p_1$ and $p_2$ and so is constant when we change $p$ along the direction $e_1 -e_2$. Changing $p$ along the direction $e_1 - e_2$ means, fixing $p \in \Delta_{d,d}$, and considering the discrete function $\phi : \{-p_1,\ldots,p_2\} \to \Z$ given by
	\[
		\phi(s) := w(p+ s e_1 - s e_2) = (p_1 + s) (p_2-s) u.
	\]
	If $u=0$, $\phi$ is identically equal to $0$. Otherwise
	it is immediately clear that $\phi$ is concave, because it is given by an expression that defines a concave quadratic polynomial. 
\end{proof}

\section{An asymptotically sharp bound derived from square inequalities}\label{S:weak}

One of the main tools in proving the asymptotically sharp bound in Theorem~\ref{T:main result} is the following inequality which expresses a log-concavity property of $\MV_K$ over a ``square'' in $\Delta_{d,d}$ whose edge directions are the standard directions 
$e_i-e_j$. 

\begin{lem}[Square Inequalities]
\label{lem:square}
Let $K \in (\Km_d)^d$ be a $d$-tuple of $d$-dimensional convex bodies. Let $u_1=e_{i_1}-e_j$ and $u_2=e_{i_2}-e_j$
 for pairwise different $i_1,i_2,j \in [d]$. Then 
\begin{align*}
\MV_K(p)\MV_K(p+u_1+u_2) \leq 2\MV_K(p+u_1)\MV_K(p+u_2),
\end{align*}
for  any $p \in  \Delta_{d,d}$ satisfying $p_j \geq 2$.
\end{lem}

\begin{proof}
This result appears in \cite[Lemma 5.1]{BGL}. For the sake of completeness we outline an argument which also appears in the proof of \cite[Lemma 7.4.1]{Schneider2014}. For simplicity we assume that $u_1=e_1-e_3$, $u_2=e_2-e_3$, and 
$p+u_1+u_2={\bm 1}$. Then, in the standard notation, the above statement becomes 
\begin{equation}\label{e:square}
\MV(K_1,K_2,K')\MV(K_3,K_3,K')\leq 2\MV(K_1,K_3,K')\MV(K_2,K_3,K'),
\end{equation}
where $K'$ denotes the $(d-2)$-tuple $(K_3,\dots, K_d)$.
Consider a family of $d$-tuples of convex bodies $(K_1+sK_3,K_2+tK_3,K')$ for positive real $s,t$. 
It follows by the Aleksandrov-Fenchel inequality applied to this tuple that the quadratic form 
\begin{align*}
At^2+2Bst+Cs^2,\quad\text{where}\quad
A&=\MV(K_1,K_3,K')^2-\MV(K_1,K_1,K')\MV(K_3,K_3,K')\\
B&=\MV(K_1,K_2,K')\MV(K_3,K_3,K')-\MV(K_1,K_3,K')\MV(K_2,K_3,K')\\
C&=\MV(K_2,K_3,K')^2-\MV(K_2,K_2,K')\MV(K_3,K_3,K')
\end{align*}
is non-negative for all positive $s,t$. Similarly, applying the Aleksandrov-Fenchel inequality to the tuple $(tK_1+sK_2,K_3,K')$ we see that 
the quadratic form $At^2-2Bst+Cs^2$ is non-negative for all positive $s,t$. This implies that the discriminant of both forms must be non-positive, i.e. $B^2\leq AC$. Ignoring the negative terms in $A$ and $C$, this produces:
$$\left(\MV(K_1,K_2,K')\MV(K_3,K_3,K')-\MV(K_1,K_3,K')\MV(K_2,K_3,K')\right)^2\leq \MV(K_1,K_3,K')^2\MV(K_2,K_3,K')^2.$$
Finally, taking the square root of both sides and rearranging, we obtain (\ref{e:square}).
\end{proof}

The square inequalities indeed give relations that do not follow as combinations of Aleksandrov-Fenchel inequalities as the following 
shows.

\begin{cor}
\label{cor:AF_notfull}
Let $d \in \Z_{\geq 3}$. There exist functions $f \colon \Delta_{d,d}
\to \R_{\geq 0}$ that satisfy all Aleksandrov-Fenchel relations but that
are not of the form $\MV_K$ for any $K \in (\Km_d)^d$.
\end{cor}

\begin{proof}
We will explicitly construct one such function $f$.
Set $f(\bm{1})=3$ and $f(p)=1$ for all $\bm{1} \neq p \in \Delta_{d,d}$.
It is easy to verify that $f$ satisfies all Aleksandrov-Fenchel relations.
However, as $3 = f(3,0,0,1\dots,1)f(1,1,1,1,\dots,1) > 
2 f(2,1,0,1\dots,1)f(2,0,1,1,\dots,1) = 2$, Lemma~\ref{lem:square}
shows that there exists no $K \in (\Km_d)^d$ that satisfies $\MV_K=f$.
\end{proof} 

For our later purposes we need a slight generalization of Lemma~\ref{lem:square}
that can be obtained by combining different square inequalities.
It is convenient to introduce the following notation. Consider a subset $I\subset [d]$ and an element $j\in[d]\setminus I$.
Denote
$$u_{I,j}=\sum_{i\in I}(e_i-e_j).$$
When $I=\{i\}$ we write $u_{i,j}$ for $u_{\{i\},j}=e_i-e_j$.

\begin{lem}[Generalized Square Inequalities]
\label{gen_square}
Let $K \in (\Km_d)^d$ be a $d$-tuple of $d$-dimensional convex bodies. Let $I\subset[d]$ and $i,j\in[d]\setminus I$. Then
\begin{align*}
\MV_K(p)\MV_K(p+u_{I,j}+u_{i,j}) 
\leq 
2^{|I|} \MV_K(p+u_{I,j})\MV_K(p+u_{i,j}).
\end{align*}
for  any $p \in \Delta_{d,d}$ satisfying $p_j > |I|$.
\end{lem}

\begin{proof}
We will prove the statement by induction on $|I|$. 
Note that for $|I|=1$ the statement is given by Lemma~\ref{lem:square}. 
Assume $|I|>1$. Pick $k\in I$ and let $I'=I\setminus\{k\}$. By the induction hypothesis
\begin{align*}
\MV_K(p)\MV_K(p+u_{I',j}+u_{i,j}) 
\leq 
2^{|I'|} \MV_K(p+u_{I',j})\MV_K(p+u_{i,j}).
\end{align*}
Applying Lemma~\ref{lem:square} where we replace $p$ by $p+u_{I',j}$ and set $u_1=u_{i,j}$ and $u_2=u_{k,j}$, we obtain
\begin{align*}
\MV_K(p+u_{I',j})\MV_K(p+u_{I',j}+u_{i,j}+u_{k,j}) \leq 2\MV_K(p+u_{I',j}+u_{i,j})\MV_K(p+u_{I',j}+u_{k,j}).
\end{align*}
Multiplying the above two inequalities and noting that $u_{I',j}+u_{k,j}=u_{I,j}$ we obtain the claim.
\end{proof}

The following lemma shows that the functions $\MV_K\in\R^{\Delta_{d,d}}$ 
satisfy certain weak $\log$-concavity relations 
in any direction of the form $u_{I,j}$ for $I\subset [d]$ and $j\in[d]\setminus I$.

\begin{lem}
\label{conc_dir}
Let  $K \in (\Km_d)^d$ be a $d$-tuple of $d$-dimensional convex bodies.
Let $I\subset[d]$ and $j\in[d]\setminus I$. Then
$$\MV_K(p+ku_{I,j})^{\frac{l}{k+l}}\MV_K(p-lu_{I,j})^{\frac{k}{k+l}}\leq 2^{kl{|I|\choose 2}}\MV_K(p)$$
for any $k,l\in\N$ and $p\in \Delta_{d,d}$ satisfying $p+ku_{I,j}, p-lu_{I,j}\in\Delta_{d,d}$.
\end{lem}

\begin{proof}
We will prove the special case of $k=l=1$ and the general case follows from (\ref{e:restate-iii}) in Remark~\ref{r:restate}.
The proof of the special case is again via induction on $|I|$. For $|I|=1$ we recover the Aleksandrov-Fenchel inequality. 

Assume $|I|>1$. Pick $i\in I$ and let $I'=I\setminus\{i\}$. Then $u_{I,j}=u_{i,j}+u_{I',j}$.
By the induction hypothesis, replacing $p$ by $p+u_{i,j}$, we have
$$\MV_K(p+u_{i,j}+u_{I',j})^{\frac{1}{2}}\MV_K(p+u_{i,j}-u_{I',j})^{\frac{1}{2}}\leq 2^{{|I'|\choose 2}}\MV_K(p+u_{i,j}).$$
Furthermore, by the Aleksandrov-Fenchel inequality we have
$$\MV_K(p+u_{i,j}-u_{I',j})^{\frac{1}{2}}\MV_K(p-u_{i,j}-u_{I',j})^{\frac{1}{2}}\leq \MV_K(p-u_{I',j}).$$
Finally, by Lemma~\ref{gen_square}, where we replace $I$ by $I'$ and $p$ by $p-u_{I',j}$, we have
\begin{align*}
\MV_K(p-u_{I',j})\MV_K(p+u_{i,j}) 
\leq 
2^{|I'|} \MV_K(p)\MV_K(p+u_{i,j}-u_{I',j}).
\end{align*}
It remains to multiply the three inequalities above and note that ${|I'|\choose 2}+|I'|={|I|\choose 2}$.
\end{proof}

Our next result (Theorem~\ref{thm:steps}) provides a method for bounding mixed volumes in directions of the form $\sum_{i\in I} e_{i}-\sum_{j\in J} e_{j}$ 
for some disjoint subsets $I, J\subset[d]$ with $|I|=|J|$. Similar to above we introduce special notation for such directions:
$$u_{I,J}=\sum_{i\in I} e_{i}-\sum_{j\in J} e_{j}.$$
We will first illustrate the statement and the proof of Theorem~\ref{thm:steps} with an example.

\begin{ex}\label{ex:3-d example} Let $K \in (\Km_{6,1})^6$ be a $6$-tuple of
$6$-dimensional convex bodies of volume at least $1$. We will show that 
\begin{equation}\label{e:bound}
\mv_K(2,2,2,0,0,0)\leq 2\mv_K(\bm{1})+6, 
\end{equation}
where $\bm{1}=(1,1,1,1,1,1)$. First, by Lemma~\ref{lem:concavity_simplices}, in logarithmic notation we have
\begin{equation}\label{e:first-a}
\frac{1}{3}\left(\mv_K(1,1,1,3,0,0)+\mv_K(1,1,1,0,3,0)+\mv_K(1,1,1,0,0,3)\right) \leq \mv_K(\bm{1}).
\end{equation}

The corresponding 2-simplex is depicted in blue in Figure~\ref{fig:3-d example}. Now, for each of the summands in the left hand side of 
(\ref{e:first-a}), we use the weak $\log$-concavity relations
in the directions $(1,1,1,-3,0,0)$, $(1,1,1,0,-3,0)$, and $(1,1,1,0,0,-3)$
given by Lemma~\ref{conc_dir} and obtain 
\begin{align*}
\frac{1}{2}\mv_K(2,2,2,0,0,0)+\frac{1}{2}\mv_K(0,0,0,6,0,0) & \leq 3+\mv_K(1,1,1,3,0,0)\\
\frac{1}{2}\mv_K(2,2,2,0,0,0)+\frac{1}{2}\mv_K(0,0,0,0,6,0) & \leq 3+\mv_K(1,1,1,0,3,0)\\
\frac{1}{2}\mv_K(2,2,2,0,0,0)+\frac{1}{2}\mv_K(0,0,0,0,0,6) & \leq 3+\mv_K(1,1,1,0,0,3).
\end{align*}
In Figure~\ref{fig:3-d example} these directions are shown in green.
These inequalities, together with (\ref{e:first-a}), provide the bound (\ref{e:bound}), as $\mv_K(0,0,0,6,0,0)$, $\mv_K(0,0,0,0,6,0)$, and $\mv_K(0,0,0,0,0,6)$ are non-negative.
\end{ex}

\begin{figure}
\includegraphics[scale=0.45]{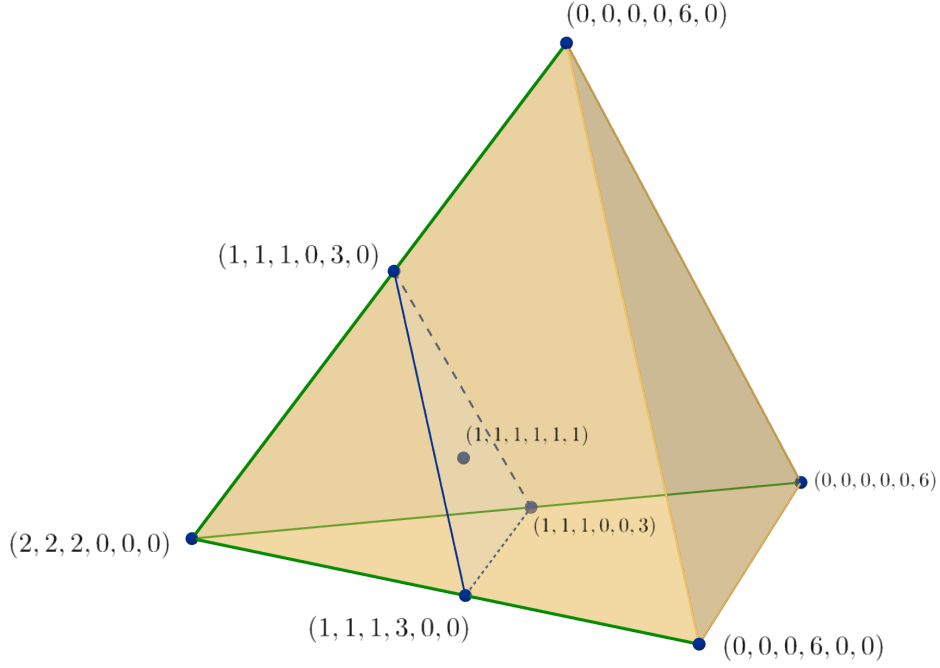}
\caption{Bounding $\mv_K(2,2,2,0,0,0)$ in terms of $\mv_K(1,1,1,1,1,1)$. We use the fact that all points that we draw live inside 
the $3$-dimensional slice 
$\{p \in \Delta_{6,6} \colon p_1=p_2=p_3\}$ of the 
$5$-dimensional simplex $\Delta_{6,6}$.}
\label{fig:3-d example}
\end{figure}

\begin{thm}
\label{thm:steps}
Let $K \in (\Km_{d,1})^d$ be a $d$-tuple of $d$-dimensional convex bodies
of volume at least~$1$. Let
$I,J\subset[d]$ be disjoint subsets with $|I|=|J|$.
Then 
$$
\mv_K(p+u_{I,J}) \leq 
\frac{\mu+1}{\mu} 
\mv_K(p) + (\mu+1){\lfloor d/2\rfloor\choose 2}, 
$$
for any $p \in \Delta_{d,d}$ such that $p\pm u_{I,J}\in\Delta_{d,d}$, 
where $\mu = \min(p_{i} : i\in I)$.
\end{thm}

\begin{proof} 
First we write $p$ as the barycenter of a simplex with vertices $b_j=p-u_{J\setminus\{j\},j}$, for $j\in J$.
Applying Lemma~\ref{lem:concavity_simplices} we get
\begin{align}
\label{e:step}
\frac{1}{|J|}\sum_{j\in J} \mv_K(p-u_{J\setminus\{j\},j})\leq \mv_K(p).
\end{align}
In order to establish the required bound for $\mv_K(p+u_{I,J})$ we estimate each summand $\mv_K(p-u_{J\setminus\{j\},j})$ from below using the weak concavity relations along $u_{I,j}$ given by Lemma~\ref{conc_dir}. Indeed, applying  Lemma~\ref{conc_dir}
with $p$ replaced by $p-u_{J\setminus\{j\},j}$ and $(k,l)=(1,\mu)$, in the logarithmic notation we get
 $$
 \frac{\mu}{\mu+1}\mv_K(p-u_{J\setminus\{j\},j}+u_{I,j})+ \frac{1}{\mu+1}\mv_K(p-u_{J\setminus\{j\},j}-\mu u_{I,j})\leq \mu{|I|\choose 2}+\mv_K(p-u_{J\setminus\{j\},j}).
 $$
Note that $-u_{J\setminus\{j\},j}+u_{I,j}=u_{I,J}$. Also, since we assumed that
the $K_i$ have volume at least~1, the second term in the left-hand side is non-negative and, hence, can be dropped. We thus obtain
 $$
 \frac{\mu}{\mu+1}\mv_K(p+u_{I,J})- \mu{|I|\choose 2}\leq\mv_K(p-u_{J\setminus\{j\},j}).
 $$
Plugging these estimates into \eqref{e:step} yields 
\begin{align*}
\mv_K(p+u_{I,J}) \leq 
\frac{\mu+1}{\mu} 
\mv_K(p) + (\mu+1){|I|\choose 2}.
\end{align*}
Finally, using  $|I|\leq \lfloor d/2\rfloor$ we get the claim.
\end{proof}

\begin{rem}
One may be led to think that functions $\MV_K$ satisfy certain
weak log-concavity relations in any direction
$u_{I,J}$.  Indeed, the bound of 
Theorem~\ref{thm:steps} is exactly what one would get from
relations of the form
\begin{align*}
\frac{\mu}{\mu+1}\mv_K(p+u_{I,J})
+ \frac{1}{\mu+1}\mv_K(p-u_{I,J}) \leq  
\mv_K(p) + C. 
\end{align*} 
for $C={|I| \choose 2}$.
In Proposition~\ref{P:11-1-1} below we show that for $|I|=|J|=2$
we indeed almost have such relations but for a slightly larger
constant $C = 2$.
However, according to our computations, for 
 $|I|=|J|>2$ our methods cannot show such weak
concavity relations along $u_{I,J}$ anymore no matter the constant $C$. 
\end{rem}

\begin{prop}\label{P:11-1-1} 
Let $K \in (\Km_d)^d$ be a $d$-tuple of $d$-dimensional convex bodies and
$I,J$ disjoint subsets of $[d]$ with $|I|=|J|=2$.
Then 
$$\MV_K(p+u_{I,J})\MV_K(p-u_{I,J})\leq 2^{4}\MV_K(p)^2,$$
for any $p \in \Delta_{d,d}$ satisfying 
$p\pm u_{I,J}\in \Delta_{d,d}$.
\end{prop}
\begin{proof} 
For simplicity, we assume $I=\{1,2\}$,  $J=\{3,4\}$. 
Applying Lemma~\ref{lem:square}
with $p$ replaced by $p-e_1+e_3$ and $u_1=e_1-e_3$, $u_2=e_2-e_3$ we get
$$\MV_K(p-e_1+e_3)\MV_K(p+e_2-e_3)\leq 2\MV_K(p)\MV_K(p-e_1+e_2).$$
Next, applying Lemma~\ref{lem:square} with $p$ replaced by $p+e_1+e_2-e_3-e_4$ and  $u_1=e_3-e_1$, $u_2=e_4-e_1$ we get
$$\MV_K(p+e_1+e_2-e_3-e_4)\MV_K(p-e_1+e_2)\leq 2\MV_K(p+e_2-e_4)\MV_K(p+e_2-e_3).$$
Multiplying the above inequalities we obtain
$$\MV_K(p+e_1+e_2-e_3-e_4)\MV_K(p-e_1+e_3)\leq 4\MV_K(p)\MV_K(p+e_2-e_4).$$
Similarly, switching the indices $1\leftrightarrow 4$ and $2\leftrightarrow 3$, we obtain
$$\MV_K(p-e_1-e_2+e_3+e_4)\MV_K(p-e_4+e_2)\leq 4\MV_K(p)\MV_K(p+e_3-e_1).$$
Finally, the product of the last two inequalities provides the result.
\end{proof}

The following is our key result regarding  bounds on mixed volumes in 
general dimension.

\begin{thm}
\label{thm:asymp_Kound}
Let $K \in (\Km_{d,1})^d$ be a $d$-tuple of $d$-dimensional convex bodies
of volume at least $1$ and 
$p \in \Delta_{d,d}$. Then one has:
\begin{align}
\label{eq:sq_pbound}
\mv_K(p) \leq \max(p)\left(\mv_K(\bm{1})+(\max(p)-1){\lfloor d/2\rfloor\choose 2}\right), 
\end{align}
Consequently, 
\begin{align*}
\mv_K(p) \leq \max(p)\mv_K(\bm{1}) + C(d), 
\end{align*}
where $C(d)$ is a constant only depending on the dimension $d$.

Furthermore, given that $\MV_K(\bm{1})=m$, one 
obtains the following bound:
\begin{align}
\label{eq:sq_minkbound}
	\Vol(\Sigma(K)) \le  2^{d(d-1){\lfloor d/2\rfloor\choose 2}} d^d\,m^d.
\end{align}
\end{thm}

\begin{proof}
We will show that there is a sequence
of inequalities of the type shown in Theorem~\ref{thm:steps} that yields 
\eqref{eq:sq_pbound}. Let us, without loss of generality, assume that 
$p$ is a decreasing vector, that is $p_1\geq \dots\geq p_d$. Hence $\max(p)=p_1$.

Let us define the set of \emph{admissible vectors} 
$\mathcal{S}_p$ at a point
$p \in \Delta_{d,d}$ to be
\begin{align*}
\mathcal{S}_p \coloneqq \Big\{\sum_{i=1}^n e_{i} -\sum_{j=l+1}^{l+n}e_{j} 
\text{ for } l\geq n\geq 1, l+n\leq d 
\text{ and } n \text{ satisfying } p_1=\dots=p_n \Big\}.
\end{align*}
We claim that there is a sequence of decreasing vectors $a_1,\dots, a_{p_1}\in \Delta_{d,d}$ starting at $a_1=\bm{1}$ and ending at $a_{p_1}=p$ such that $a_{i+1}-a_i\in\mathcal{S}_{a_i}$ and, hence, 
$\max(a_i)=i$ for all $1\leq i<p_1$. We call such a sequence an {\it admissible path} from $\bm{1}$ to $p$.
The existence of such a path can be easily seen by induction on
$p_1$. If $p_1=1$ then $p = \bm{1}$ and there is nothing
to show, so let $p_1\geq 2$. 
Let $n$ be the maximal index satisfying $p_n = p_1$ and $l$
be the maximal index satisfying $p_l > 0$. Consider the 
vector 
\begin{align*}
p' = (p_1-1,\dots,p_n-1,p_{n+1},\dots,p_l,\underbrace{1,\dots,1}_{n \text{ times}},0,\dots,0).
\end{align*}
One can check that $p'\in \Delta_{d,d}$ exists and is decreasing
by construction. By the induction hypothesis 
there is an admissible path from $\bm{1}$ to $p'$ of length $p_1-1$. Moreover, $p-p'\in\mathcal{S}_{p'}$, and therefore there exists an admissible path from $\bm{1}$ to $p$ of length $p_1$. For example, Figure~\ref{fig:paths} shows all admissible paths from $\bm{1} \in 
\Delta_{6,6}$ to any decreasing point $p$ in $\Delta_{6,6}$.

Let us now show how the existence of such an admissible path implies
\eqref{eq:sq_pbound}.
Let $a_{i+1}$ and $a_i$ be two terms in an admissible path from $\bm{1}$ to $p$. By Theorem~\ref{thm:steps}
 we have 
$$
\mv_K(a_{i+1}) \leq \frac{\mu+1}{\mu} \mv_K(a_i)+(\mu+1){\lfloor d/2\rfloor\choose 2}, 
$$
where $\mu$ is the minimum of those entries of $a_i$ which increase when we pass to $a_{i+1}$. But all these entries are equal to $i$ by the construction of the admissible sequence. Hence, we can write
$$
\mv_K(a_{i+1}) \leq \frac{i+1}{i} \mv_K(a_i)+(i+1){\lfloor d/2\rfloor\choose 2}. 
$$
Applying this repeatedly we obtain
\begin{align*}
\mv_K(p)\leq\left(\frac{2}{1}\right) \left(\frac{3}{2}\right) \cdots 
\left(\frac{p_1}{p_1-1}\right)\mv_K(\bm{1})+p_1(p_1-1){\lfloor d/2\rfloor\choose 2} 
=p_1\mv_K(\bm{1})+p_1(p_1-1){\lfloor d/2\rfloor\choose 2}, 
\end{align*}
which concludes the proof of \eqref{eq:sq_pbound}. The inequality using
a constant $C(d)$ only depending on the dimension $d$ follows
directly from \eqref{eq:sq_pbound} and
the observation that $\max(p)$ is bounded by $d$. 

Assume now that $\MV_K(\bm{1})=m$. 
Combining Proposition~\ref{prop:formula_mink} with the observation
that the maximum of the bounds from \eqref{eq:sq_pbound} is attained
e.g. at $p = (d,0,\dots,0)$, one obtains
\begin{align*}
\Vol(\Sigma(K)) = \sum_{p \in \Delta_{d,d}} 
\binom{d}{p} 2^{\mv_K(p)} \leq 
\sum_{p \in \Delta_{d,d}}\binom{d}{p} 2^{\mv_K(d,0,\dots,0)}  = d^d 2^{\mv_K(d,0,\dots,0)}.
\end{align*}
Explicitly plugging in the bound from \eqref{eq:sq_pbound} for $\mv_K(d,0,\dots,0)$
yields \eqref{eq:sq_minkbound}.
\end{proof}

\begin{figure}\label{fig:paths}
\includegraphics[scale=0.7]{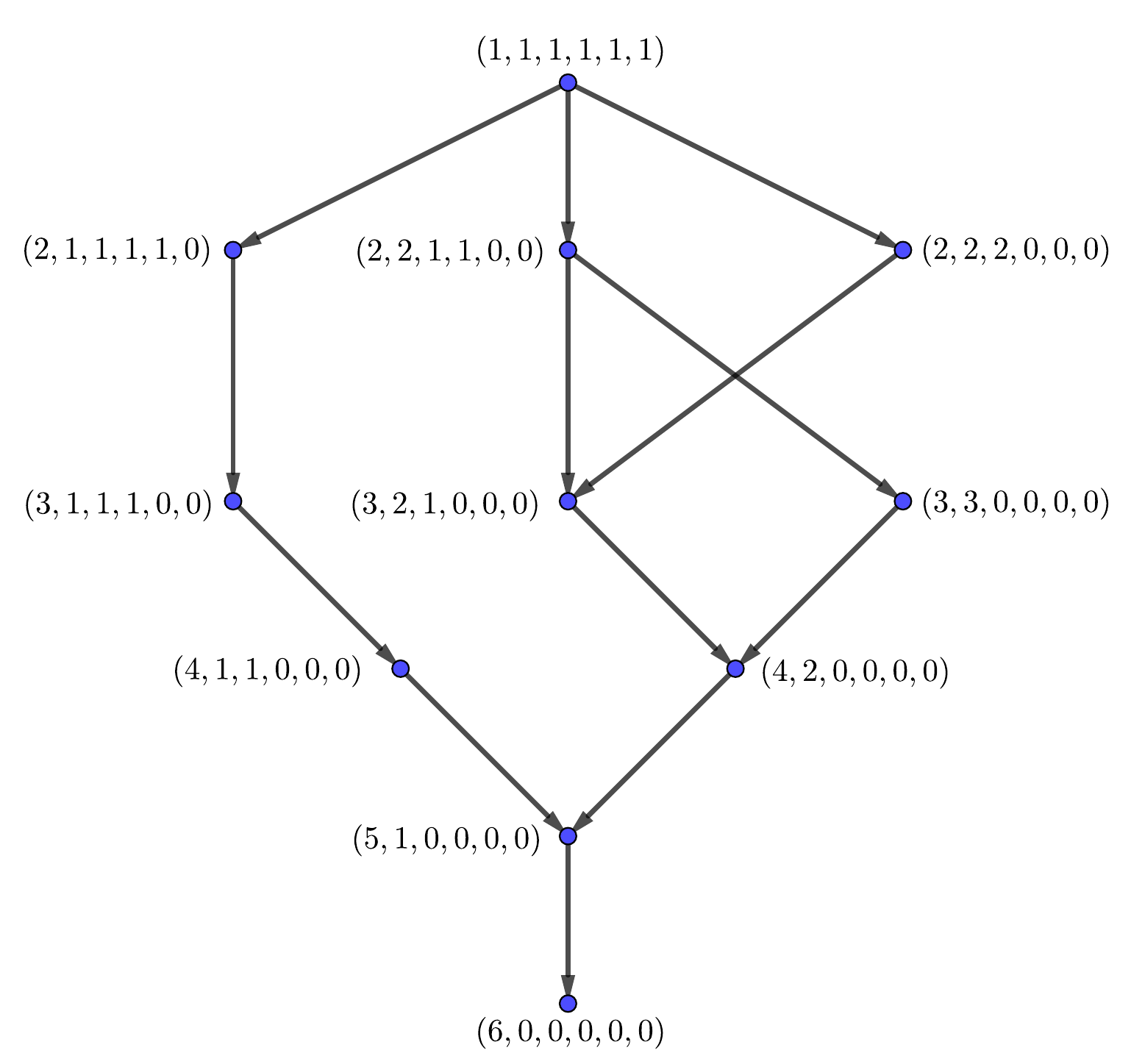}
\caption{Admissible paths from $\bm{1} \in 
\Delta_{6,6}$ to
any decreasing point $p \in \Delta_{6,6}$.}
\end{figure}

\begin{rem}
Note that the bound from Theorem~\ref{thm:asymp_Kound} shows that,
for any $p \in \Delta_{d,d}$, the maximum of $\MV_K(p)$ among all 
$d$-tuples $(\Km_{d,1})^d$ of $d$-dimensional convex bodies of volume
at least $1$ that satisfy $\MV_K(\mathbf{1})=m$ is of order 
$O(m^{\max(p)})$ as $m \rightarrow \infty$. To see that the order of
this bound is sharp, fix $p \in \Delta_{d,d}$ and let $i \in [d]$ be 
an index satisfying $p_i = \max(p)$. Then any tuple
$K \in (\Km_{d,1})^d$ of the form
$K_i = m A$ and $K_j = A$ for every $ j \in [d]\setminus\{i\}$ for a convex
body $A$ with $\Vol(A)=1$ yields $\MV_K(\mathbf{1})=m$, while
$\MV_K(p)=m^{p_i}=m^{\max(p)}$.
\end{rem}

\begin{proof}[Proof of Theorem~\ref{T:main result}]
	The assertion is a direct consequence of Theorem~\ref{thm:asymp_Kound}.
\end{proof}

\begin{proof}[Proof of Corollary~\ref{cor:systems}]
	Let $P_i$ be the Newton polytope of $f_i$. Then $Q=P_1 + \cdots + P_d$ is the Newton polytope of the product $f_1 \cdots f_d$. The number of monomials in $f_1 \cdots f_d$ is at most the number of lattice points in $Q$. By Blichfedt's inequality \cite{blichfeldt1914anew}, one has $|Q \cap \Z^d| \le \Vol(Q) + d$, which in combination with Theorem~\ref{T:main result}, yields the assertion.
\end{proof}

\section{Confirmation of Conjecture~\ref{conj} in dimension $3$} \label{sec:dim:3}

 In this section we use a  computer-assisted approach to prove Theorem~\ref{three:maxima:in:dimension:three}, which establishes Conjecture~\ref{conj} in dimension $3$.  The high level description of the approach is as follows. In the setting of Conjecture~\ref{conj}, we know that $\mv_K \in (\log m) \AF_d$.  So, we calculate the vertices of the Aleksandrov-Fenchel polytope $\AF_3$ using a computer. Since $\Vol(K_1 + \cdots + K_\ell)$ is a linear combination of mixed volumes, we conclude that $\Vol(K_1 + \cdots + K_\ell)= F(\mv_K)$, where $F$ is an explicitly given convex function.  Since $F$ is convex, the maximum of $F$  on $(\log m) \AF_d$ is attained at the vertices of $(\log m) \AF_3$. The values of $F$ at the vertices of $(\log m ) \AF_3$ are functions of $m$ given by rather simple algebraic expressions. It turns out that one can bound all such expressions from above by $(m + \ell -1)^3$ for $m \in \R_{\ge 1}$.

While the Aleksandrov-Fenchel polytope has rather many vertices (there are $24$ vertices in total), the amount of algebraic computations that we need to carry out can be significantly reduced by taking into account the symmetries. On $\R^{\Delta_{3,3}}$ we introduce the action of the symmetric group $S_3$ on three elements. We introduce the action of $S_3$ on $\R^{\Delta_{3,3}}$ by defining $\sigma v$ as
\[	(\sigma v)(p_1,p_2,p_3) = v(p_{\sigma(1)},p_{\sigma(2)},p_{\sigma(3)})
\]
for $\sigma \in S_3$ and $v \in \R^{\Delta_{3,3}}$.  
It is clear that $\AF_3$ is invariant under the action of $S_3$ on $\R^{\Delta_{3,3}}$, which means that $\sigma v \in \AF_3$ holds for all $\sigma \in S_3$ and all $v \in \AF_3$. 

In the following proposition, we use $e_p$ with $p \in \Delta_{d,d}$ to denote the standard basis vectors of $\R^{\Delta_{d,d}}$. This means, $e_p(q)  \in \{0,1\}$ with $e_p(q)=1$ if and only if $p=q$.\label{key} 

\begin{prop}[Vertices of $\AF_3$]
\label{prop:af_vertices}
The polytope $\AF_3$ has $24$ vertices, 
which are split into $7$ orbits under the action of $S_3$ on $\AF_3$, with the orbits generated by the following seven vertices
\begin{align*}
v_1 = & e_{(1,1,1)}, \\
v_2 = & 2e_{(2,1,0)}+e_{(1,2,0)}+e_{(1,1,1)}, \\
v_3 = &2e_{(2,1,0)}+2e_{(1,2,0)}+e_{(1,1,1)}, \\
v_4 = &2e_{(2,1,0)}+e_{(1,2,0)}+\frac{1}{2} e_{(2,0,1)}
+ e_{(1,0,2)} + e_{(1,1,1)}, \\
v_5 = &2e_{(2,1,0)}+e_{(1,2,0)}+2e_{(2,0,1)}
+ e_{(1,0,2)} + e_{(1,1,1)}, \\
v_6 = &2e_{(2,1,0)}+e_{(1,2,0)}+2e_{(2,0,1)}
+ e_{(1,0,2)} + 3e_{(3,0,0)} +e_{(1,1,1)}, 
\\
v_7 = &\frac{2}{3}e_{(2,1,0)}+\frac{4}{3}e_{(1,2,0)}+\frac{4}{3}e_{(2,0,1)}
+ \frac{2}{3}e_{(1,0,2)}+
\frac{2}{3}e_{(0,2,1)} + 
\frac{4}{3}e_{(0,1,2)} + e_{(1,1,1)}.
\end{align*}
\end{prop}
\begin{proof}
	We used sagemath \cite{sagemath} to determine the vertices of $\AF_3$, given by a system of linear inequalities. Sagemath is one of the many possibilities to do computations with polytopes over the field of rational numbers. Polymake is yet another possibility. 
\end{proof}

\begin{figure}
\begin{center}
	$v_1$
	\begin{tikzpicture}[scale=0.5,baseline=1cm]
	\draw (0.0,0.0) -- (3.0,5.196152422706632);
	\draw (0.0,0.0) -- (6.0,0.0);
	\draw (0.0,0.0) -- (0.0,0.0);
	\draw (2.0,0.0) -- (4.0,3.464101615137755);
	\draw (1.0,1.7320508075688774) -- (5.0,1.7320508075688774);
	\draw (2.0,0.0) -- (1.0,1.7320508075688774);
	\draw (4.0,0.0) -- (5.0,1.7320508075688774);
	\draw (2.0,3.464101615137755) -- (4.0,3.464101615137755);
	\draw (4.0,0.0) -- (2.0,3.464101615137755);
	\draw (6.0,0.0) -- (6.0,0.0);
	\draw (3.0,5.196152422706632) -- (3.0,5.196152422706632);
	\draw (6.0,0.0) -- (3.0,5.196152422706632);
	\filldraw[fill=white] (0.0,0.0) circle (0.5);
	\node (0-0-3) at (0.0,0.0) {0};
	\filldraw[fill=white] (1.0,1.7320508075688774) circle (0.5);
	\node (0-1-2) at (1.0,1.7320508075688774) {0};
	\filldraw[fill=white] (2.0,3.464101615137755) circle (0.5);
	\node (0-2-1) at (2.0,3.464101615137755) {0};
	\filldraw[fill=white] (3.0,5.196152422706632) circle (0.5);
	\node (0-3-0) at (3.0,5.196152422706632) {0};
	\filldraw[fill=white] (2.0,0.0) circle (0.5);
	\node (1-0-2) at (2.0,0.0) {0};
	\filldraw[fill=white] (3.0,1.7320508075688774) circle (0.5);
	\node (1-1-1) at (3.0,1.7320508075688774) {1};
	\filldraw[fill=white] (4.0,3.464101615137755) circle (0.5);
	\node (1-2-0) at (4.0,3.464101615137755) {0};
	\filldraw[fill=white] (4.0,0.0) circle (0.5);
	\node (2-0-1) at (4.0,0.0) {0};
	\filldraw[fill=white] (5.0,1.7320508075688774) circle (0.5);
	\node (2-1-0) at (5.0,1.7320508075688774) {0};
	\filldraw[fill=white] (6.0,0.0) circle (0.5);
	\node (3-0-0) at (6.0,0.0) {0};
	\end{tikzpicture}
	\hspace{5mm}
	$v_2$
	\begin{tikzpicture}[scale=0.5,baseline=1cm]
	\draw (0.0,0.0) -- (3.0,5.196152422706632);
	\draw (0.0,0.0) -- (6.0,0.0);
	\draw (0.0,0.0) -- (0.0,0.0);
	\draw (2.0,0.0) -- (4.0,3.464101615137755);
	\draw (1.0,1.7320508075688774) -- (5.0,1.7320508075688774);
	\draw (2.0,0.0) -- (1.0,1.7320508075688774);
	\draw (4.0,0.0) -- (5.0,1.7320508075688774);
	\draw (2.0,3.464101615137755) -- (4.0,3.464101615137755);
	\draw (4.0,0.0) -- (2.0,3.464101615137755);
	\draw (6.0,0.0) -- (6.0,0.0);
	\draw (3.0,5.196152422706632) -- (3.0,5.196152422706632);
	\draw (6.0,0.0) -- (3.0,5.196152422706632);
	\filldraw[fill=white] (0.0,0.0) circle (0.5);
	\node (0-0-3) at (0.0,0.0) {0};
	\filldraw[fill=white] (1.0,1.7320508075688774) circle (0.5);
	\node (0-1-2) at (1.0,1.7320508075688774) {0};
	\filldraw[fill=white] (2.0,3.464101615137755) circle (0.5);
	\node (0-2-1) at (2.0,3.464101615137755) {0};
	\filldraw[fill=white] (3.0,5.196152422706632) circle (0.5);
	\node (0-3-0) at (3.0,5.196152422706632) {0};
	\filldraw[fill=white] (2.0,0.0) circle (0.5);
	\node (1-0-2) at (2.0,0.0) {0};
	\filldraw[fill=white] (3.0,1.7320508075688774) circle (0.5);
	\node (1-1-1) at (3.0,1.7320508075688774) {1};
	\filldraw[fill=white] (4.0,3.464101615137755) circle (0.5);
	\node (1-2-0) at (4.0,3.464101615137755) {1};
	\filldraw[fill=white] (4.0,0.0) circle (0.5);
	\node (2-0-1) at (4.0,0.0) {0};
	\filldraw[fill=white] (5.0,1.7320508075688774) circle (0.5);
	\node (2-1-0) at (5.0,1.7320508075688774) {2};
	\filldraw[fill=white] (6.0,0.0) circle (0.5);
	\node (3-0-0) at (6.0,0.0) {0};
	\end{tikzpicture}
	\vspace{5mm}
   $v_3$
	\begin{tikzpicture}[scale=0.5,baseline=1cm]
	\draw (0.0,0.0) -- (3.0,5.196152422706632);
	\draw (0.0,0.0) -- (6.0,0.0);
	\draw (0.0,0.0) -- (0.0,0.0);
	\draw (2.0,0.0) -- (4.0,3.464101615137755);
	\draw (1.0,1.7320508075688774) -- (5.0,1.7320508075688774);
	\draw (2.0,0.0) -- (1.0,1.7320508075688774);
	\draw (4.0,0.0) -- (5.0,1.7320508075688774);
	\draw (2.0,3.464101615137755) -- (4.0,3.464101615137755);
	\draw (4.0,0.0) -- (2.0,3.464101615137755);
	\draw (6.0,0.0) -- (6.0,0.0);
	\draw (3.0,5.196152422706632) -- (3.0,5.196152422706632);
	\draw (6.0,0.0) -- (3.0,5.196152422706632);
	\filldraw[fill=white] (0.0,0.0) circle (0.5);
	\node (0-0-3) at (0.0,0.0) {0};
	\filldraw[fill=white] (1.0,1.7320508075688774) circle (0.5);
	\node (0-1-2) at (1.0,1.7320508075688774) {0};
	\filldraw[fill=white] (2.0,3.464101615137755) circle (0.5);
	\node (0-2-1) at (2.0,3.464101615137755) {0};
	\filldraw[fill=white] (3.0,5.196152422706632) circle (0.5);
	\node (0-3-0) at (3.0,5.196152422706632) {0};
	\filldraw[fill=white] (2.0,0.0) circle (0.5);
	\node (1-0-2) at (2.0,0.0) {0};
	\filldraw[fill=white] (3.0,1.7320508075688774) circle (0.5);
	\node (1-1-1) at (3.0,1.7320508075688774) {1};
	\filldraw[fill=white] (4.0,3.464101615137755) circle (0.5);
	\node (1-2-0) at (4.0,3.464101615137755) {2};
	\filldraw[fill=white] (4.0,0.0) circle (0.5);
	\node (2-0-1) at (4.0,0.0) {0};
	\filldraw[fill=white] (5.0,1.7320508075688774) circle (0.5);
	\node (2-1-0) at (5.0,1.7320508075688774) {2};
	\filldraw[fill=white] (6.0,0.0) circle (0.5);
	\node (3-0-0) at (6.0,0.0) {0};
	\end{tikzpicture}
	\\
	$v_4$
	\begin{tikzpicture}[scale=0.5,baseline=1cm]
	\draw (0.0,0.0) -- (3.0,5.196152422706632);
	\draw (0.0,0.0) -- (6.0,0.0);
	\draw (0.0,0.0) -- (0.0,0.0);
	\draw (2.0,0.0) -- (4.0,3.464101615137755);
	\draw (1.0,1.7320508075688774) -- (5.0,1.7320508075688774);
	\draw (2.0,0.0) -- (1.0,1.7320508075688774);
	\draw (4.0,0.0) -- (5.0,1.7320508075688774);
	\draw (2.0,3.464101615137755) -- (4.0,3.464101615137755);
	\draw (4.0,0.0) -- (2.0,3.464101615137755);
	\draw (6.0,0.0) -- (6.0,0.0);
	\draw (3.0,5.196152422706632) -- (3.0,5.196152422706632);
	\draw (6.0,0.0) -- (3.0,5.196152422706632);
	\filldraw[fill=white] (0.0,0.0) circle (0.5);
	\node (0-0-3) at (0.0,0.0) {0};
	\filldraw[fill=white] (1.0,1.7320508075688774) circle (0.5);
	\node (0-1-2) at (1.0,1.7320508075688774) {0};
	\filldraw[fill=white] (2.0,3.464101615137755) circle (0.5);
	\node (0-2-1) at (2.0,3.464101615137755) {0};
	\filldraw[fill=white] (3.0,5.196152422706632) circle (0.5);
	\node (0-3-0) at (3.0,5.196152422706632) {0};
	\filldraw[fill=white] (2.0,0.0) circle (0.5);
	\node (1-0-2) at (2.0,0.0) {1};
	\filldraw[fill=white] (3.0,1.7320508075688774) circle (0.5);
	\node (1-1-1) at (3.0,1.7320508075688774) {1};
	\filldraw[fill=white] (4.0,3.464101615137755) circle (0.5);
	\node (1-2-0) at (4.0,3.464101615137755) {1};
	\filldraw[fill=white] (4.0,0.0) circle (0.5);
	\node (2-0-1) at (4.0,0.0) {$\frac{1}{2}$};
	\filldraw[fill=white] (5.0,1.7320508075688774) circle (0.5);
	\node (2-1-0) at (5.0,1.7320508075688774) {2};
	\filldraw[fill=white] (6.0,0.0) circle (0.5);
	\node (3-0-0) at (6.0,0.0) {0};
	\end{tikzpicture}
	\hspace{5mm}
	$v_5$
	\begin{tikzpicture}[scale=0.5,baseline=1cm]
	\draw (0.0,0.0) -- (3.0,5.196152422706632);
	\draw (0.0,0.0) -- (6.0,0.0);
	\draw (0.0,0.0) -- (0.0,0.0);
	\draw (2.0,0.0) -- (4.0,3.464101615137755);
	\draw (1.0,1.7320508075688774) -- (5.0,1.7320508075688774);
	\draw (2.0,0.0) -- (1.0,1.7320508075688774);
	\draw (4.0,0.0) -- (5.0,1.7320508075688774);
	\draw (2.0,3.464101615137755) -- (4.0,3.464101615137755);
	\draw (4.0,0.0) -- (2.0,3.464101615137755);
	\draw (6.0,0.0) -- (6.0,0.0);
	\draw (3.0,5.196152422706632) -- (3.0,5.196152422706632);
	\draw (6.0,0.0) -- (3.0,5.196152422706632);
	\filldraw[fill=white] (0.0,0.0) circle (0.5);
	\node (0-0-3) at (0.0,0.0) {0};
	\filldraw[fill=white] (1.0,1.7320508075688774) circle (0.5);
	\node (0-1-2) at (1.0,1.7320508075688774) {0};
	\filldraw[fill=white] (2.0,3.464101615137755) circle (0.5);
	\node (0-2-1) at (2.0,3.464101615137755) {0};
	\filldraw[fill=white] (3.0,5.196152422706632) circle (0.5);
	\node (0-3-0) at (3.0,5.196152422706632) {0};
	\filldraw[fill=white] (2.0,0.0) circle (0.5);
	\node (1-0-2) at (2.0,0.0) {1};
	\filldraw[fill=white] (3.0,1.7320508075688774) circle (0.5);
	\node (1-1-1) at (3.0,1.7320508075688774) {1};
	\filldraw[fill=white] (4.0,3.464101615137755) circle (0.5);
	\node (1-2-0) at (4.0,3.464101615137755) {1};
	\filldraw[fill=white] (4.0,0.0) circle (0.5);
	\node (2-0-1) at (4.0,0.0) {2};
	\filldraw[fill=white] (5.0,1.7320508075688774) circle (0.5);
	\node (2-1-0) at (5.0,1.7320508075688774) {2};
	\filldraw[fill=white] (6.0,0.0) circle (0.5);
	\node (3-0-0) at (6.0,0.0) {0};
	\end{tikzpicture}
	\hspace{5mm}
	$v_6$
	\begin{tikzpicture}[scale=0.5,baseline=1cm]
	\draw (0.0,0.0) -- (3.0,5.196152422706632);
	\draw (0.0,0.0) -- (6.0,0.0);
	\draw (0.0,0.0) -- (0.0,0.0);
	\draw (2.0,0.0) -- (4.0,3.464101615137755);
	\draw (1.0,1.7320508075688774) -- (5.0,1.7320508075688774);
	\draw (2.0,0.0) -- (1.0,1.7320508075688774);
	\draw (4.0,0.0) -- (5.0,1.7320508075688774);
	\draw (2.0,3.464101615137755) -- (4.0,3.464101615137755);
	\draw (4.0,0.0) -- (2.0,3.464101615137755);
	\draw (6.0,0.0) -- (6.0,0.0);
	\draw (3.0,5.196152422706632) -- (3.0,5.196152422706632);
	\draw (6.0,0.0) -- (3.0,5.196152422706632);
	\filldraw[fill=white] (0.0,0.0) circle (0.5);
	\node (0-0-3) at (0.0,0.0) {0};
	\filldraw[fill=white] (1.0,1.7320508075688774) circle (0.5);
	\node (0-1-2) at (1.0,1.7320508075688774) {0};
	\filldraw[fill=white] (2.0,3.464101615137755) circle (0.5);
	\node (0-2-1) at (2.0,3.464101615137755) {0};
	\filldraw[fill=white] (3.0,5.196152422706632) circle (0.5);
	\node (0-3-0) at (3.0,5.196152422706632) {0};
	\filldraw[fill=white] (2.0,0.0) circle (0.5);
	\node (1-0-2) at (2.0,0.0) {1};
	\filldraw[fill=white] (3.0,1.7320508075688774) circle (0.5);
	\node (1-1-1) at (3.0,1.7320508075688774) {1};
	\filldraw[fill=white] (4.0,3.464101615137755) circle (0.5);
	\node (1-2-0) at (4.0,3.464101615137755) {1};
	\filldraw[fill=white] (4.0,0.0) circle (0.5);
	\node (2-0-1) at (4.0,0.0) {2};
	\filldraw[fill=white] (5.0,1.7320508075688774) circle (0.5);
	\node (2-1-0) at (5.0,1.7320508075688774) {2};
	\filldraw[fill=white] (6.0,0.0) circle (0.5);
	\node (3-0-0) at (6.0,0.0) {3};
	\end{tikzpicture}
	\\
	\vspace{3mm}
	$v_7$
	\begin{tikzpicture}[scale=0.5,baseline=1cm]
	\draw (0.0,0.0) -- (3.0,5.196152422706632);
	\draw (0.0,0.0) -- (6.0,0.0);
	\draw (0.0,0.0) -- (0.0,0.0);
	\draw (2.0,0.0) -- (4.0,3.464101615137755);
	\draw (1.0,1.7320508075688774) -- (5.0,1.7320508075688774);
	\draw (2.0,0.0) -- (1.0,1.7320508075688774);
	\draw (4.0,0.0) -- (5.0,1.7320508075688774);
	\draw (2.0,3.464101615137755) -- (4.0,3.464101615137755);
	\draw (4.0,0.0) -- (2.0,3.464101615137755);
	\draw (6.0,0.0) -- (6.0,0.0);
	\draw (3.0,5.196152422706632) -- (3.0,5.196152422706632);
	\draw (6.0,0.0) -- (3.0,5.196152422706632);
	\filldraw[fill=white] (0.0,0.0) circle (0.5);
	\node (0-0-3) at (0.0,0.0) {0};
	\filldraw[fill=white] (1.0,1.7320508075688774) circle (0.5);
	\node (0-1-2) at (1.0,1.7320508075688774) {$\frac{4}{3}$};
	\filldraw[fill=white] (2.0,3.464101615137755) circle (0.5);
	\node (0-2-1) at (2.0,3.464101615137755) {0};
	\filldraw[fill=white] (3.0,5.196152422706632) circle (0.5);
	\node (0-3-0) at (3.0,5.196152422706632) {0};
	\filldraw[fill=white] (2.0,0.0) circle (0.5);
	\node (1-0-2) at (2.0,0.0) {$\frac{2}{3}$};
	\filldraw[fill=white] (3.0,1.7320508075688774) circle (0.5);
	\node (1-1-1) at (3.0,1.7320508075688774) {1};
	\filldraw[fill=white] (4.0,3.464101615137755) circle (0.5);
	\node (1-2-0) at (4.0,3.464101615137755) {$\frac{4}{3}$};
	\filldraw[fill=white] (4.0,0.0) circle (0.5);
	\node (2-0-1) at (4.0,0.0) {$\frac{4}{3}$};
	\filldraw[fill=white] (5.0,1.7320508075688774) circle (0.5);
	\node (2-1-0) at (5.0,1.7320508075688774) {$\frac{2}{3}$};
	\filldraw[fill=white] (6.0,0.0) circle (0.5);
	\node (3-0-0) at (6.0,0.0) {0};
	\end{tikzpicture}
\end{center}	
\caption{Illustration to Proposition~\ref{prop:af_vertices}. The Aleksandrov-Fenchel polytope $\AF_3$ has $24$ vertices that are split into $7$ orbits under the action of $S_3$. The diagrams present the coordinates $v_i(p)$ of the seven vertices $v_1,\ldots,v_7$.}
\end{figure}
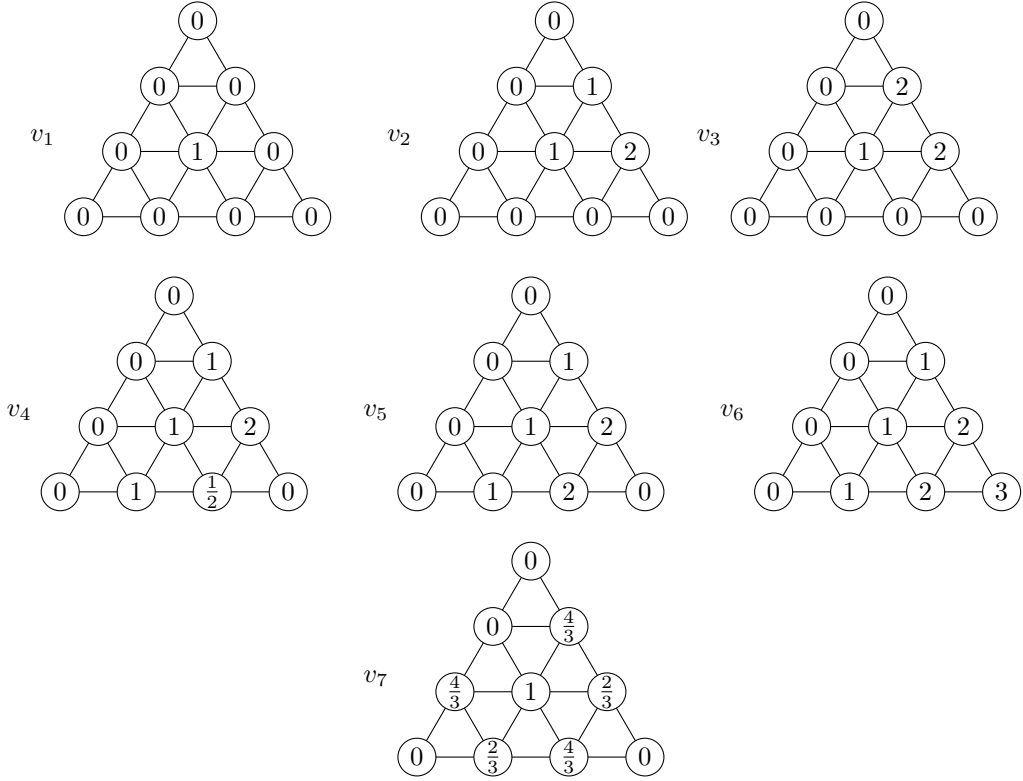

\begin{proof}[Proof of Theorem~\ref{three:maxima:in:dimension:three}]
	For all three assertions, the equality case is verified in a straightforward way. We prove the respective inequalities. 
	
	By Remark~\ref{R:single-volume-bound}, $\Vol(K_1) \le m^3$, so \eqref{dim:3:1} follows.
	For the verification of assertions \eqref{dim:3:2} and \eqref{dim:3:3}, we use Proposition~\ref{prop:af_vertices}. We fix the standard component-wise partial order $\le $ on 
	$\R^{\Delta_{3,3}}$, 
	that is, $\mv \le \mw$  if and only if $\mv(p) \le \mw(p)$ holds for every $p \in \Delta_{3,3}$. 
	It is clear that the vertices $v_1,\ldots,v_6$ of $\AF_3$ are related by 
	\begin{align}
	& v_1 \le v_2 \le v_3  \label{eq:v123}
	\\ & v_4 \le v_5 \le v_6 \label{eq:v456}
	\end{align}
	
	For \eqref{dim:3:2} we have 
	\begin{align*}
		\Vol(K_1 + K_2) & 
		= \sum_{i=0}^3 \binom{3}{i} \MV_K(i,3-i,0) 
		 = \sum_{i=0}^3 \binom{3}{i} 2^{{\mv_K(i,3-i,0)}},
	\end{align*}
	where ${\mv_K}\in (\log m)\AF_{3}$. 
	Changing the base from 2 to $m$, we see that 
	$\Vol(K_1 + K_2)$ is bounded by the maximum of the function $g_m : \R^{\Delta_{3,3}} \to \R$
\[
g_m(\mv) := m^{\mv(3,0,0)} + 3 m^{\mv(2,1,0)} + 3 m^{\mv(1,2,0)} + m^{\mv(0,3,0)}
\]
over $\mv \in \AF_3$. The function $g_m(\mv)$ is convex so that the maximum is attained at one of the vertices of $\AF_3$. By Proposition~\ref{prop:af_vertices}, the vertices of $\AF_3$ have the form $\sigma v_i$ with $\sigma \in S_3$ and $i \in \{1,\ldots,7\}$. Taking into account \eqref{eq:v123} and \eqref{eq:v456}, it follows that it is enough to check the cases $i \in \{3,6,7\}$. First, we detect the maximum of $g_m$ in the orbits generated by $v_3, v_6$ and $v_7$. It is straightforward to check that
\begin{align*}
\phi_3(m):=\max_{\sigma \in S_3} f_m(\sigma v_3) & = 2 + 6 m^2,
\\ \phi_6(m):= \max_{\sigma \in S_3} f_m(\sigma v_6) & = 1 + 3 m + 3 m^2  + m^3 = (m+1)^3
\\ \phi_7(m):= \max_{\sigma \in S_3} f_m(\sigma v_7) & = 2 + 3 m^{2/3} + 3 m^{4/3}.
\end{align*}
Clearly, $\phi_7(m) \le \phi_3(m) \le \phi_6(m)$, where $\phi_3(m) \le \phi_6(m)$ holds since $\phi_6(m) - \phi_3(m) = (m-1)^3$. 
Thus, $(m+1)^3$ is an upper bound for $\Vol(K_1+K_2)$.

	Similarly, for \eqref{dim:3:3} we have  
	$$\Vol(K_1+ K_2 + K_3) = \sum_{p \in \Delta_{3,3}} \binom{3}{p} \MV_K(p) =\sum_{p \in \Delta_{3,3}} \binom{3}{p} 2^{{\mv_K(p)}},$$ 
 where $\mv_K\in (\log m)\AF_3$. To obtain the desired upper bound for $\Vol(K_1 + K_2 + K_3)$ we maximize the function $f_m : \R^{\Delta_{3,3}} \to \R$ 
	\[
		f_m(\mv) := \sum_{p \in \Delta_{3,3}} \binom{3}{p} m^{\mv(p)}
	\]
	over $\mv \in \AF_3$. Again, the function $f_m$ is convex and so its maximum is necessarily attained in one of the vertices of $\AF_3$. On the other hand, it is clear that the function is invariant under the action of $S_3$ on $\AF_3$, as one clearly has $f_m(\sigma v) = f_m(v)$ for every $v \in \AF_3$ and $\sigma \in S_3$. It follows that it is enough to compare the values of $f_m$ on the vertices $v_1,\ldots,v_7$ from Proposition~\ref{prop:af_vertices}. That means $\mv \le \mw$ implies $f_m(\mv) \le f_m(\mw)$ for all $\mv,\mw \in \R^{\Delta_{3,3}}$. The latter property follows from the assumption $m \ge 1$ and the non-negativity of multinomial coefficients. In view of \eqref{eq:v123} and \eqref{eq:v456} 
	it suffices to compare $f_m(v_3), f_m(v_6)$ and $f_m(v_7)$. The non-negativity of $f_m(v_6) - f_m(v_3)$ for $m \ge 1$ can be phrased as the non-negativity of $f_{m+1}(v_6) - f_{m+1}(v_3)$ for $m \ge 0$. It turns out that $f_{m+1}(v_6) - f_{m+1}(v_3)$ is a polynomial in $m$ all of whose  coefficients are non-negative. Hence $f_{m+1}(v_6) - f_{m+1}(v_3) \ge 0$ holds for every $m \ge 0$, which implies $f_m(v_6) - f_m(v_3) \ge 0$ for $m\geq 1$. 
	
	Comparing $f_m(v_7)$ to $f_m(v_6)$ can be carried out in a similar fashion, but note that $v_7$ is a fractional point. We can still reduce the verification to the polynomial setting by noticing that $3 v_7$ is an integral point. 	
	The validity of $f_m(v_6) \ge f_m(v_7)$ for all $m \ge 1$ can be rephrased as the inequality $f_{(m+1)^3}(v_6) - f_{(m+1)^3}(v_7) \ge 0$ for all $m \ge 0$. The latter is true since $f_{(m+1)^3}(v_6) - f_{(m+1)^3}(v_7)$ is a polynomial  all of whose coefficients are non-negative. Summarizing, we conclude that $f_m(v_6) = (m+2)^3$ is the maximum of $f_m(\mv)$ for $\mv \in \AF_3$ and, hence, an upper bound on $\Vol(K_1 + K_2 + K_3)$.
\end{proof}

\section{Concluding remarks and outlook} \label{sec:remarks}

\subsection{On tuples maximizing the volume of the Minkowski sum}

 The following proposition converts Conjecture~\ref{conj} to a more specific situation. 
\begin{prop}\label{prop:specific}
	Let $m \in \R_{\ge 1}$ and let $\ell \in \{1,\ldots,d\}$. Consider a tuple $K= (K_1,\ldots,K_d)$ of convex bodies  satisfying  
	\begin{align*}
	\Vol(K_1) \ge 1, \ldots, \Vol(K_d) \ge 1,\quad\text{and }\ \V(K_1,\ldots,K_d) = m
	\end{align*}
	and maximizing $\Vol(K_1 + \cdots + K_\ell)$. Then 
	\begin{enumerate}
		\item \label{specific:Vol=1} For each such optimal tuple, $\Vol(K_i) = 1$ holds for all except possibly one choice of $i \in \{1,\ldots, \ell\}$ and for every $i > \ell$. 
		\item \label{specific:equality} For $\ell < d-1$, there exists an optimal tuple that satisfies $K_{\ell+1} = \cdots = K_d$. 
	\end{enumerate} 
\end{prop} 

\begin{proof} 
	\eqref{specific:Vol=1}
	If $\alpha_i := \Vol(K_i)^{1/d} > 1$ holds for some $i > \ell$ then the tuple is not optimal since changing $K_1$ to $\alpha_i K_1$ and $K_i$ to $\frac{1}{\alpha_i} K_i$ we obtain a new tuple $K'= (K'_1,\ldots,K'_d)$
	of mixed volume $m$ with $\Vol(K'_1 + \cdots K'_\ell) > \Vol(K_1 + \cdots + K_\ell)$. 
	
	Now, assume that $\ell \geq 2$ and that for at least two choices of $i \in \{1,\ldots, \ell\}$ one has $\Vol(K_i) > 1$. We can assume $\Vol(K_1) > 1$ and $\Vol(K_2) > 1$.
	We consider the tuple $( \frac{1}{t} K_1, t K_2, K_3,\ldots, K_d)$, depending on $t > 0$. Clearly, $\MV(\frac{1}{t} K_1, t K_2,K_3,\ldots, K_d) = \MV(K)$. Furthermore,  the function $f : \R_{> 0} \to \R_{>0}$ given by $f(t) := \Vol( \frac{1}{t} K_1 + t K_2 + K_3 + \cdots + K_\ell)$ is a strictly convex function. This can be seen by writing $f(t)$ as a non-negative linear combination of functions $t^p$, which are strictly convex for every $p \in \R \setminus \{0\}$. For $\epsilon>0$ small enough and every $t \in [1-\epsilon,1+ \epsilon]$, the volumes of $\frac{1}{t} K_1$ and $t K_2$ are at least one. Since $f(t)$ is strictly convex, its maximum on $[1-\epsilon,1+\epsilon]$ is attained at the boundary and is strictly larger than $f(0)$.  This contradicts the optimality of the tuple $K$ and shows that $\Vol(K_i)=1$ for all except possible one choice of $i \in \{1,\ldots,\ell\}$.

	\eqref{specific:equality} In view of Lemma~\ref{lem:concavity_simplices}, 
	\begin{align*}
		m = V(K_1,\ldots,K_d) 
		& \ge \prod_{i=\ell+1}^d V(K_1,\ldots,K_\ell, K_i,\dots, K_i)^{\frac{1}{d-\ell}} 
		\\ & \ge \min_{i \in \{\ell+1,\ldots,d\}} V(K_1,\ldots,K_\ell, K_i,\ldots, K_i) =: m'. 
	\end{align*}

	So, taking the $i$ for which the above minimum is attained and replacing the tuple $(K_1,\ldots,K_d)$ by the tuple $( \frac{m}{m'} K_1,K_2,\ldots, K_\ell, K_i,\ldots, K_i)$, we keep the mixed volume of the tuple unchanged without decreasing the volume of the Minkowski sum of its first $\ell$ bodies. 
\end{proof}

The latter proposition somewhat simplifies the original optimization problem, but still the problem remains non-trivial. Say, for $d=4$ and $\ell=2$, the problem is turned to the maximization of $\Vol(A+B)$ subject to $\Vol(A) \ge 1, \Vol(B)=\Vol(C)=1$ and  $\MV(A,B,C,C)=m$.

\subsection{Relations between mixed volumes: the quest for tight inequalities and a complete description.}\label{S:relations}
The work on the problem of bounding $\Vol(\Sigma(K))$ has taught us that the current knowledge of the relations between mixed volumes is still rather limited and the literature might miss some important inequalities beyond the classical ones. Such new inequalities would probably be of interest to a broader community of experts, including researchers interested in metric aspects of convex sets, as well as researchers working on combinatorial aspects of algebraic geometry. 
The problem of describing the relationship between mixed volumes 
goes back to the 1960 work \cite{Shephard1960} of Shephard (see also Problems~6.1 in \cite[p.~109]{gruber2007convex} for a similar problem for the so-called Quermassintegrals).

In \cite{Shephard1960} Shephard provided a complete description of mixed volume configurations for two $d$-dimensional convex bodies.
Recall that $\Km_d$ denotes the family of all $d$-dimensional convex bodies in $\R^d$.

\begin{thm}[{Shephard \cite[Thm.~4]{Shephard1960}}]
	\label{thm:shep}
	The mixed-volume configuration space
	$\MV(\Km_d, \Delta_{2,d})$
	is the set of all $\V \in \R_{>0}^{\Delta_{2,d}}$ that satisfy the Aleksandrov-Fenchel inequalities
	\begin{align*}
	\V(i,d-i)^2 & \ge \V(i+1,d-i-1) \V(i-1,d-i+1) \quad \forall i \in [d-1].  
	\end{align*}
	Equivalently, the logarithmic mixed volume configuration space $\mv(\Km_d,\Delta_{2,d})$ is a polyhedral cone, described by the linearized Aleksandrov-Fenchel inequalities
	\[
	2 \mv(i,d-i) \ge \mv(i+1,d-i-1) + \mv(i-1,d-i+1) \quad \forall i \in [d-1].
	\]
\end{thm}
A refined version of Theorem~\ref{thm:shep} can be found in \cite[Lemma~2.1]{henk2012steiner}. This brings us to the following natural question about  mixed volume configuration spaces in general.



\begin{prob}\label{prob:semialg}
	Let $n, d \in \Z_{\ge 2}$ and
	let $\Km$ be the family of all compact convex sets in $\R^d$. Is $\MV(\Km,\Delta_{n,d})$ a semialgebraic set? That is, can $\MV(\Km,\Delta_{n,d})$ be described by a boolean combination of polynomial inequalities?
\end{prob}

Problem~\ref{prob:semialg} is open for all choices of $n$ and $d$ except for the case $n=2$, covered by Theorem~\ref{thm:shep}, and the case $(n,d)= (3,2)$, solved by Heine \cite{heine1938wertvorrat}.  To formulate the result of Heine we consider the family $\Km$ non-empty compact convex subsets of $\R^2$. With each triple $K=(K_1,K_2,K_3) \in \Km^3$ of such sets one can associate the matrix
\[
M_K := \begin{pmatrix} \MV(K_1,K_1) & \MV(K_1,K_2) & \MV(K_1,K_3) 
\\ \MV(K_2,K_1) & \MV(K_2,K_2) & \MV(K_2,K_3)
\\ \MV(K_3,K_1) & \MV(K_3,K_2) & \MV(K_3,K_3)
\end{pmatrix}.
\] 
The matrix $M_K$ is symmetric and has non-negative entries. Clearly, $\V(\Km,\Delta_{3,2})$ is linearly isomorphic to $\setcond{ M_K }{K \in \Km^3}$. By \eqref{AFI}, $\MV(K_i,K_j)^2 \ge \MV(K_i,K_i) \MV(K_j,K_j)$ holds for all $1 \le i < j \le 3$, which means 
the three diagonal minors of $M_K$ are non-positive. It turns out that these conditions are not enough to describe the respective mixed volume configuration, because there is yet another inequality $\det(M_K) \ge 0$, which is missing. As was shown by Heine, adding this inequality, ones obtains a complete description:

\begin{thm}[{Heine \cite[p.\,118]{heine1938wertvorrat}}]
	\label{thm:heine}
	Let $\Km$ be the family of compact convex subsets of $\R^2$. Then $\setcond{ M_K }{K \in \Km^3}$ is the set of symmetric $3 \times 3$ matrices with non-negative entries that satisfy the conditions
	\begin{align*}
	\det(M) & \ge 0, 
	& \det(M_{\{1,2\}}) & \le 0, 
	& \det(M_{\{1,3\}}) & \le 0, 
	& \det(M_{\{2,3\}}) & \le 0.
	\end{align*}
	Here, $\det(M_I)$ is the diagonal minor indexed by $I \subseteq \{1,2,3\}$. 
\end{thm}
The condition $\det(M) \ge 0$ in Theorem~\ref{thm:heine} is non-redundant. Consider, for example, the matrix
\[
M = \begin{pmatrix} 1 & 1 & 2
\\ 1 & 1 & 1
\\ 2 & 1 & 1
\end{pmatrix}
\] 
with $\det(M)  = -1, \ \det(M_{\{1,2\}}) = 0, \ \det(M_{\{1,3\}}) = -3, \  \det(M_{\{2,3\}})=0$, for which all of the conditions but $\det(M) \ge 0$ are fulfilled.  
By Theorem~\ref{thm:heine}, Problem~\ref{prob:semialg} has a positive solution for $n=3$ and $d=2$, as it provides an explicit description of $\V(\Km,\Delta_{3,2})$ by a system of non-strict polynomial inequalities. The smallest open cases of the classification problem for $\V(\Km,\Delta_{n,d})$ are $(n,d)=(4,2)$ and $(n,d)=(3,3)$. In view of Heine's theorem, already in dimension $2$, \eqref{AFI} does not provide all possible relations between mixed volumes. As a complement, our result clearly indicates that, in dimension at least five, \eqref{AFI} does not even provide the correct asymptotic approximation of relations between mixed volumes.

\bibliographystyle{amsalpha}
\bibliography{lit}

\end{document}